\documentclass[a4paper,12pt]{article}

\usepackage{amsthm,amsmath,stmaryrd,bbm,hyperref,geometry,color,authblk}
\usepackage[utf8]{inputenc}
\usepackage{amssymb}
\usepackage[english]{babel}
\usepackage{graphicx}
\usepackage{amsfonts,amssymb}
\usepackage{verbatim}
\usepackage{enumitem}

\newcommand{\po}{\left(}
\newcommand{\pf}{\right)}
\newcommand{\co}{\left[}
\newcommand{\cf}{\right]}
\newcommand{\cco}{\llbracket}
\newcommand{\ccf}{\rrbracket}
\newcommand{\R}{\mathbb R}

\newcommand{\N}{\mathbb N} 
\newcommand{\Sd}{\mathbb S^{d-1}}
 
\newcommand{\dd}{\text{d}}

\newcommand{\na}{\nabla}
\newcommand{\1}{\mathbbm{1}} 
\newcommand{\unif}{\overline}

\newtheorem{theorem}{Theorem}
\newtheorem{assumption}{Assumption}
\newtheorem{lemma}[theorem]{Lemma}

\newtheorem{corollary}[theorem]{Corollary}
\newtheorem{proposition}[theorem]{Proposition}
\newtheorem{remark}{Remark}
\newtheorem{example}{Example}

\title{Long-time $L^p$ Wasserstein contraction for diffusion processes without global dissipativity}
\author{Pierre Monmarché}

\begin{document}
\maketitle

\begin{abstract}
The fact that a Markov diffusion semi-group on $\R^d$ contracts the $L^p$ Wasserstein distance, which  has been extensively used to establish uniform-in-time stability estimates (e.g. with respect to numerical discretization errors), is a well-studied question in the case where the distances are in fact deterministically contracted by the drift (global dissipativity condition) or in the case $p=1$ (with reflection couplings). This work focuses on the non-globally dissipative case with $p>1$. This situation was previously considered in \cite{MonmarcheBruit}, but only for elliptic processes, and with a restriction on the diffusivity coefficient (which had to be large enough). Here, we extend this analysis to non-elliptic processes and provide sharper conditions to get contractions along synchronous coupling, including negative results, lower bounds and a characterization (at least in dimension 1) in terms of the maximal eigenvalue of a Feynman-Kac operator.
\end{abstract}

\section{Introduction and main results}

\subsection{Motivation and difficulties}

Consider on $\R^d$ a diffusion process solving 
\begin{equation}
\label{eq:EDS_X}
\dd X_t  = b(X_t) \dd t + \sigma \dd B_t
\end{equation}
with $b\in\mathcal C^1(\R^d,\R^d)$,  $\sigma\in\R^{d\times d}$ a constant matrix and $B$ an Brownian motion in $\R^d$. Some of our results could be extended to cases where $\sigma$ depends on $X$, but since this is not the case of the models we are interested in, for the sake of clarity we won't consider this generalization. However, since we are interested in kinetic processes (see Section~\ref{sec:applications}), let us emphasize that $\sigma$ may be singular (in fact, our main general result applies  when $\sigma=0$, which is nevertheless not the case of interest).

Denote by $(P_t)_{t\geqslant 0}$ the Markov semi-group associated to the process, i.e. $P_t f(x) = \mathbb E_x(f(X_t))$ for measurable bounded $f$, and write $\nu P_t$ the law of $X_t$ starting from  $X_0\sim \nu$.

A semimetric on $\R^d$ is a function $\omega:\R^d\times \R^d \rightarrow \R_+$ with $\omega(x,y)=\omega(y,x)$ for all $x,y\in\R^d$ and such that $\omega(x,y)=0$ if and only if $x=y$. For such a semimetric and two probability measures $\nu,\mu$ over $\R^d$, we write
\[\mathcal W_{\omega}(\nu,\mu) = \inf_{\pi\in\Pi(\nu,\mu)} \int_{\R^d\times\R^d} \omega(x,y)\pi(\dd x,\dd y)\,,\]
where $\Pi(\nu,\mu)$ stands for the set of transport plans between $\nu$ and $\mu$. When $\omega(x,y)=|x-y|^p$ for some $p\geqslant 1$ we write $\mathcal W_p(\nu,\mu) = \po \mathcal W_{\omega}(\nu,\mu)\pf^{1/p}$ the usual $L^p$ Wasserstein distance. If $\omega$ is such that $\omega(x,y) \leqslant C \po 1 +|x-y|^p \pf $ for some $C>0$, $p\geqslant 1$, then $\mathcal W_{\omega}(\nu,\mu)<\infty$ for all $\nu,\mu\in\mathcal P_p(\R^d)$ (the set of probability measures on $\R^d$ with finite $p^{th}$ moment).

The main goal of this work is to give some criteria to ensure that, for a given $p\geqslant 1$, $P_t$ is a contraction of $\mathcal W_p$ for $t$ large enough, i.e.
\begin{equation}
\label{eq:contraction_def1}
\forall \nu,\mu \in \mathcal P_p(\R^d)\,,\qquad \mathcal W_p(\nu P_{t_*},\mu P_{t_*}) \leqslant \gamma \mathcal W_p(\nu,\mu)
\end{equation}
for some $t_*>0$ and $\gamma\in[0,1)$. Assuming that $b$ is one-sided Lipschitz (see Assumtion~\ref{Assum:basic}),  by the semi-group property, this is equivalent to the existence of $C,\lambda>0$ such that 
\begin{equation}
\label{eq:contraction_def}
\forall \nu,\mu \in \mathcal P_p(\R^d)\,, \forall t\geqslant 0,\qquad \mathcal W_p(\nu P_{t},\mu P_{t}) \leqslant C e^{-\lambda t}\mathcal W_p(\nu,\mu)\,.
\end{equation}
By conditioning with respect to the initial condition, we can equivalently consider only initial conditions $\nu=\delta_x,\mu=\delta_y$ for $x,y\in\R^d$. In other words, we are looking for conditions under which
\[\beta_p(t) := \sup_{\nu,\mu\in\mathcal P_p(\R^d), \nu\neq \mu} \frac{\mathcal W_p(\nu P_t,\mu P_t)}{\mathcal W_p(\nu,\mu)} = \sup_{x,y\in\R^d,x\neq y} \frac{\mathcal W_p(\delta_x P_t,\delta_y P_t)}{|x-y|}\,,\]
which may be referred to as a $\mathcal W_p$-Dobrushin contraction coefficient \cite{DelMoral} or an $L^p$ coarse Ricci curvature \cite{JoulinOllivier}, goes to zero at infinity (equivalently, is less than $1$ for $t$ large enough). By Jensen's inequality, $\mathcal W_p \leqslant \mathcal W_q$ when $q\geqslant p$, and thus $p\mapsto \beta_p(t)$ is non-decreasing for all $t\geqslant 0$. In other words, the larger $p$, the stronger the $\mathcal W_p$-contraction. Note that we are specifically interested in the standard Wasserstein distances associated to the Euclidean distance, and in particular we won't consider Lyapunov-based weights as in e.g. \cite{EberleGuillinZimmer,DelMoral}.

Contractions of the form~\eqref{eq:contraction_def} have a lot of important consequences, the first one being that, by completeness of $(\mathcal P_p(\R^d),\mathcal W_p)$, it implies by  the Banach fixed-point theorem (together with the semi-group property) that $(P_t)_{t\geqslant 0}$ admits a unique invariant measure $\mu_\infty \in \mathcal P_p(\R^d)$ and that $\nu P_t $ converges to $\mu_\infty$ exponentially fast in $\mathcal W_p$ for all $\nu\in\mathcal P_p(\R^d)$. However, we emphasize that a contraction~\eqref{eq:contraction_def} is a priori much stronger than the exponential contraction towards the invariant measure in the $\mathcal W_p$ sense, which would be~\eqref{eq:contraction_def} but only with $\mu = \mu_\infty$. For instance, in the case where $\sigma=\sqrt{2/\beta}\mathrm{Id}$ for some $\beta>0$ and $b(x) = -\na U(x)$, a $\mathcal W_2$ contraction towards $\mu_\infty \propto \exp(-\beta U)$ is known to hold when $\mu_\infty$ satisfies a  so-called log-Sobolev inequality (see e.g. \cite[Theorem 2.1.]{ren2021exponential} or similarly \cite[Theorem 2]{MonmarcheGuillinVFP} in a kinetic case), but  a global contraction~\eqref{eq:contraction_def} with $p=2$ is only known when $U$  is strongly  convex \cite{MonmarcheContraction}  or when $U$ is strongly convex outside a compact set and the temperature $\beta^{-1}$ is large enough \cite{MonmarcheBruit} (and the discussion in Section~\ref{sec:numerique} suggests that it could fail at low temperature). Beyond convergence to equilibrium, contractions like~\eqref{eq:contraction_def} have been extensively used to compare two different processes, for instance the diffusion~\eqref{eq:EDS_X} and a numerical time-discretization \cite{M25,10.1214/23-AAP2034,dalalyan2022bounding,cheng2018underdamped}, or a system of interacting particles and $N$ independent McKean-Vlasov diffusions, or two McKean-Vlasov diffusions \cite{schuh2024global,durmus2020elementary,malrieu2003convergence,bolley2010trend,M15}, etc. In these situations, it is then sufficient to analyse the difference between the two processes (e.g. numerical error analysis or propagation of chaos) in the $\mathcal W_p$ sense over the interval $[0,t_*]$, and then~\eqref{eq:contraction_def1} with the triangular inequality automatically improves these finite-time estimates to uniform-in-time ones. These arguments really rely on the contraction and, for instance, do not follow from a bound of the form 
\begin{equation}
\label{loc:Walpha}
\forall \nu,\mu \in \mathcal P_p(\R^d)\,, \forall t\geqslant 0,\qquad \mathcal W_p(\nu P_{t},\mu P_{t}) \leqslant C e^{-\lambda t}\po \mathcal W_p(\nu,\mu) + \mathcal W_p^{1/p}(\nu,\mu)\pf\,,
\end{equation}
for some $C,\lambda>0$, which is the kind of bounds proven by establishing a contraction of $\mathcal W_\omega$ with a semimetric $\omega$ with $\omega(x,y)$ of order $|x-y|^p$ for large values of $|x-y|$ and of order $|x-y|$ for small values, as in \cite{luo2016exponential} (see also Theorem~\ref{thm:contractionMarkov} and its proof).

On this topic, the case $p=1$ has a very particular status. Indeed, in this situation,~\eqref{eq:contraction_def} has been established in very general situations using Markovian couplings based on reflection or hybrid reflection/synchronous couplings, with seminal papers by Eberle and then with co-authors in the elliptic case \cite{Eberle2,Eberle1} and the kinetic one~\cite{EberleGuillinZimmer}, and many variations and applications \cite{EberleMajka,DiscreteSticky,schuh2024global}. For this reason, the present work   focuses on the case $p>1$ (although we will be interested in the case $p=1$ as an intermediary step to obtain a result for other $p$'s, see Section~\ref{subsec:non-asymp}).  A case of particular interest is $p=2$ for the following reasons among others: using Wang-Harnack inequalities, bounds in terms of $\mathcal W_2$ can often be strengthen to estimates in relative entropy, see e.g. \cite{rockner2010log,MonmarcheBruit} in the elliptic case and \cite{guillin2012degenerate,MonmarcheGuillinVFP} in the kinetic case; the Langevin process  can be interpreted as a gradient descent of the relative entropy with respect to the $\mathcal W_2$ metric, in the sense of gradient flows in metric spaces \cite{ambrosio2005gradient}; a contraction~\eqref{eq:contraction_def} with $p=2$ is equivalent by \cite{Kuwada1} to gradient bounds
\begin{equation}
\label{eq:naPt2}
|\na P_t f|^2 \leqslant C e^{-\lambda t} P_t(|\na f|^2)\,,
\end{equation}
for all Lipschitz functions $f$, which can be used to establish a Poincaré and log-Sobolev  inequalities for $\mu_\infty$ \cite{MonmarcheBruit} (see also Remark~\ref{rem:LSI}), to show a cut-off phenomenon for the process \cite{salez2025modern} or to construct Lipschitz transport maps from some measures to $\mu_\infty$ \cite{lopez2025bakry}.

 Notice that an important specificity of the case $p=1$ is that it possible to design a semimetric $\omega(x,y)$ equivalent to $|x-y|$ which  is a concave function of $|x-y|$ (for small values). As a consequence, considering two processes $X$ and $Y$, stochastic fluctuations of $|X_t-Y_t|$ (for small values) tend to decrease $\omega(X_t,Y_t)$ on average. By contrast, any semimetric equivalent to $|x-y|^p$ with $p>1$ for small values should be strongly convex (at least, on average) for small values of $|x-y|$, leading stochastic fluctuations of $|X_t-Y_t|$ to increase the average of $\omega(X_t,Y_t)$. 

This is not an issue when considering a synchronous coupling, namely two solutions of~\eqref{eq:EDS_X} driven   by the same Brownian motion with different initial condition, since in that case $|X_t-Y_t|$ has quadratic variation zero. In particular, this shows that~\eqref{eq:contraction_def} holds (simultaneously for all $p\geqslant 1$) with $C=1$ and $\lambda = \lambda^*$ where
\begin{equation}\label{eq:Lambda*}
\lambda^* =  -\sup\{v\cdot \na b(x) v, x\in\R^d,v\in\Sd\}
\end{equation}
 (with $\na b$ the Jacobian matrix of $b$), to which we refer as the Bakry-Emery curvature, see \cite{MonmarcheContraction}. This is in fact an equivalence: if~\eqref{eq:contraction_def} holds with $C=1$ for some $p\geqslant 1$, then it holds for all $p\geqslant 1$ (with $C=1$ and the same $\lambda$) and necessarily $\lambda\leqslant \lambda^*$, as a consequence of \cite[Theorem 1]{MonmarcheContraction} (see also Proposition~\ref{prop:tempscourt} below).  More generally, working up to a linear change of variable, using a synchronous coupling also shows that~\eqref{eq:contraction_def} holds (for all $p\geqslant 1$) for some $C_M\geqslant 1$ with $\lambda_M^* := \inf\{-v\cdot M\na b(x) v, x,v\in\R^d, v\cdot M v=1\}$ for any positive definite symmetric matrix $M$.  However, this is a contraction only if $\lambda_M^*>0$ for some $M$ (we refer to this as a global dissipativity condition, following e.g. \cite{M72} and references within) and this is a very restrictive condition. For instance, if $b(x)=-\na U(x)$ with a non-convex $U$, considering $x_0\in\R^d$ such that $\na^2 U(x_0)$ has a negative eigenvalue $-\eta_0$  then $\lambda_M^* \leqslant -\eta_0 < 0$ for any $M$ (in fact when $\na b$ is symmetric there is no gain in taking $M$ different from the identity, but this is useful in the general case, see \cite{MonmarcheContraction} and the non-elliptic examples in Section~\ref{sec:applications}).

\medskip

In summary, the objective of the present work is to provide criteria to establish a contraction~\eqref{eq:contraction_def1} which apply with $p>1$ without a global dissipativity condition. Ideally, we want constructive non-asymptotic bounds, and not only qualitative results stating that $\beta_p(t)$ vanishes as $t\rightarrow \infty$.

\medskip

We addressed previously this question in \cite{MonmarcheBruit}, which to our knowledge is the only work which provides  results of this type (see also \cite{LMM} for the adaptation of \cite{MonmarcheBruit} to Euler schemes of~\eqref{eq:EDS_X}). The method of~\cite{MonmarcheBruit} is to consider a synchronous coupling, but to work with a semimetric of the form $\omega(x,y) = (1+\theta(x)+\theta(y))|x-y|^2 $ for some suitable positive bounded $\theta$, instead of a Euclidean metric or a concave modification of the latter as in the globally dissipative and the $p=1$ cases. The present work also focuses on synchronous couplings but the analysis is based on different arguments, which allow to treat non-elliptic processes (by contrast to \cite{MonmarcheBruit} which  only covers cases with non-singular $\sigma$). Moreover, while~\cite{MonmarcheBruit} only establishes upper-bounds on the contraction rate, we identify a Lyapunov exponent which exactly describe whether contraction occurs along synchronous couplings and we provide several lower bounds on the contraction rate. Besides, the rate of \cite{MonmarcheBruit} is based on a rough non-local lower-bound on the local curvature $k(x) $ while we use a sharper local bound $\eta(x)$ given in~\eqref{eq:eta}.

\medskip

Let us summarize our contributions:

\begin{enumerate}
\item In Proposition~\ref{prop:COntractionkappap}, we establish essentially that the worse contraction rate along a synchronous coupling is always attained for arbitrarily close initial conditions, which means that the question boils down to the study of the (sign of the) Lyapunov exponent of a linearized problem (cf.~\eqref{eq:ContractLyapunov}).
\item We establish two ``negative" results, namely that, under basic conditions, it is not possible to have better than the Bakry-Emery curvature if we want a result which works for small times (see Proposition~\ref{prop:tempscourt}) or simultaneously for all $p \geqslant 1$ (see Proposition~\ref{prop:negativeelliptic}).
\item Theorem~\ref{thm:FerreStoltz} gives a criterion of long-time contraction along synchronous couplings in terms of the sign of a suitable cumulant, given either in a variational form or as the leading eigenvalue of a Feynman-Kac operator. This pave the way to numerical estimations, as illustrated in Section~\ref{sec:numerique}.
\item Under more restrictive conditions (which, in the case $b=-\na U$ for instance, are allowing $\na^2 U$ to have negative eigenvalues but small with respect to the other constants in the problem), more  quantitative convergence rates are given in Theorem~\ref{thm:final}. To get this result, we first establish a general result (Theorem~\ref{thm:contractionMarkov}), of interest by itself, based on Markovian (hybrid reflection/synchronous) couplings in the spirit of \cite{EberleGuillinZimmer}, which gives a contraction for a modified cost and then for $\mathcal W_1$. This induces concentration inequalities which are then used to control directly the cumulant.
\end{enumerate}

The rest of this work is organized as follows. The settings are detailed in Section~\ref{sec:generalresult}, where the main general result (Proposition~\ref{prop:COntractionkappap}), which relates the Wasserstein contraction to the study of a Lyapunov exponent, is stated. The negative results are stated in Section~\ref{subsec:negativeResults}. Section~\ref{sec:criteria} is devoted to the analysis of a cumulant whose negativity implies an $L^p$ Wasserstein contraction. More precisely, relying on large deviation results, Section~\ref{subsec:resultFerreStoltz} gives a characterization of this cumulant in terms of a variational problem (this is Theorems~\ref{thm:FerreStoltz}), which is used for numerical simulations in Section~\ref{sec:numerique}, while Section~\ref{subsec:non-asymp} present the constructive result based on an explicit Markovian coupling (Theorem~\ref{thm:contractionMarkov}) and its consequence for $L^p$ contraction (Theorem~\ref{thm:final}). Examples of applications are provided in Section~\ref{sec:applications}, illustrating the limitations of the general results when addressing standard non-elliptic processes. The proofs of the general results are gathered in Section~\ref{sec:proofGeneral}, while Section~\ref{sec:proofCoupling} gathers the proofs based on the constructive coupling. Finally, several questions which seems interesting to us are raised along the analysis and left open, and we repeat them as a conclusion in Section~\ref{sec:conclusion}.

\subsection{Settings and a general result}\label{sec:generalresult} 

The following conditions are enforced throughout the work:

\begin{assumption}[Basic conditions]\label{Assum:basic}
The drift  $b$ is in $\mathcal C^2(\R^d,\R^d)$ with bounded second-order derivatives, and there exists $L>0$ such that $(x-y)\cdot (b(x)-b(y)) \leqslant L|x-y|^2$ for all $x,y\in\R^d$.
\end{assumption}

These conditions ensure the strong well-posedness of~\eqref{eq:EDS_X}. The condition on the second-order derivatives of $b$ is made for simplicity and could be removed in many situations by approximation, as $\|\na^2 b\|_\infty$ will not be involved in the final estimates. The existence of $L$ is equivalent to the fact that $\lambda^*>-\infty$ (with the Bakry-Emery curvature given in~\eqref{eq:Lambda*}).

 For a given  solution $X=(X_t)_{t\geqslant 0}$ of~\eqref{eq:EDS_X} and $t\geqslant 0$ we write $\Phi_t^X$ the $d\times d$ random matrix such that for all $v\in\Sd$, $\Phi_t^X v = V_t$ where $(V_s)_{s\geqslant 0}$ solves 
\begin{equation}
\label{eq:EDO_Phi}
\dd V_t  = \na b(X_t) V_t \dd t\,,\qquad V_0=v \,.
\end{equation}
 For $p\geqslant 1$ and $t\geqslant 0$, introduce 
\[\kappa_p(t) = \sup_{x\in\R^d,v\in\Sd}  \po \mathbb E_x\po |\Phi_t^Xv|^p\pf\pf^{1/p}\,.\]
By direct computations, 
\begin{equation}
\label{eq:courbureBE}
\kappa_p(t) \leqslant e^{-  \lambda^* t}
\end{equation}
 for all $t\geqslant 0$ and $p\geqslant 1$, corresponding to the classical Bakry-Emery curvature bound (see Propositions~\ref{prop:tempscourt} and~\ref{prop:sharppinfini} about the sharpness of this bound). By conditionning at time $s\geqslant 0$ and the Markov property,
\begin{equation}
\label{eq:sub-multipliticity}
\kappa_p(t+s) \leqslant \kappa_p(t)\kappa_p(s)\,. 
\end{equation}
By subadditivity, this shows that the $L^p$ Lyapunov exponent
\begin{equation}
\label{eq:Lambdap*}
\Lambda_p^*:= \lim_{t\rightarrow \infty} \frac{1}{t} \ln \kappa_p(t)
\end{equation}
exists (a priori in $[-\infty,-\lambda^*]$) and is equal to $\inf_{t\geqslant 0} \frac{1}{t} \ln \kappa_p(t)$. In particular,
\begin{equation}
\label{eq:ContractLyapunov}
\Lambda_p^* < 0 \qquad \Leftrightarrow\qquad \exists t_*>0,\qquad \kappa_p(t_*)<1\,.
\end{equation}
The relation of this condition with the Wasserstein contraction~\eqref{eq:contraction_def1} is direct by the following:

\begin{proposition}\label{prop:COntractionkappap}
Under Assumption~\ref{Assum:basic},  for all $p\geqslant 1$ and $t\geqslant 0$,
\[\kappa_p^p(t) = \sup_{x,y\in\R^d,x\neq y} \frac{\mathbb E_{(x,y)}\po |X_t-Y_t|^p \pf}{|x-y|^p}\,,\]
where, in the right hand side, $(X_t,Y_t)_{t\geqslant 0}$ are two solutions of~\eqref{eq:EDS_X} (driven by the same Brownian motion $B$) with $(X_0,Y_0)=(x,y)$. In particular, for any $t\geqslant 0$, given two solutions of~\eqref{eq:EDS_X} (with the same Brownian motion),
\[\mathbb E \po |X_t-Y_t|^p\pf \leqslant \kappa_p^p(t) \mathbb E \po |X_0-Y_0|^p\pf\,.\]
As a consequence, for all $p\geqslant 1$ and $t\geqslant 0$,  $\beta_p(t) \leqslant \kappa_p(t)$.
\end{proposition}

The proof is given in Section~\ref{sec:proofProp1}. In dimension 1 with $p=1$, the bound is sharp:

\begin{proposition}\label{prop:dim1}
When $d=1$, $\kappa_1(t) = \beta_1(t)$. 
\end{proposition}

The proof is given in Section~\ref{sec:proofProp1}.

\begin{remark}
A specificity in dimension 1 (when $\sigma>0$) is that the process is always reversible with respect to the measure $\mu_\infty$ with density proportional to $e^{-U}$ with $U(x) = -2\sigma^{-2}\int_0^xb(s)\dd s$. Moreover, as discussed further in Section~\ref{subsec:resultFerreStoltz} (see also \cite{lelievre2024using} and references within) in this situation and when $p=1$, $2\Lambda_1^*/\sigma^2$ is the maximal eigenvalue of the Feynman-Kac operator
\[\mathcal L_{FK} f(x) =  f''(x) - U'(x) f'(x) - U''(x) f(x)\,,\]
which is a self-adjoint negative  operator in $L^2(\mu_\infty)$ since, by integration by parts,
\[\int_{\R^d} g \mathcal L_{FK} f \mu_\infty = -\int_{\R^d} \po g' - U' g \pf \po f' - U' f\pf \mu_\infty\,,\]
for all smooth $f,g$. As a consequence, $\Lambda_1^*\leqslant 0$.
\end{remark}

\begin{remark}\label{rem:torus}
An arguably unexpected (even in dimension $1$) consequence of Propositions~\ref{prop:COntractionkappap} and~\ref{prop:dim1} is that, for $p=d=1$, the optimal contraction rate is always attained for the synchronous coupling of two processes~\eqref{eq:EDS_X}, although known proofs for contraction in non-globally-dissipative situations rely on the reflection coupling (except for~\cite{MonmarcheBruit} but the latter requires $\sigma^2$ to be large enough). Notice that this is clearly false in the periodic torus: it is known in that case that the Brownian motion induces a $\mathcal W_1$-contraction (using e.g. a reflection coupling), while the parallel coupling gives $|X_t-Y_t|=|X_0-Y_0|$ for all $t\geqslant 0$. Whether the Brownian motion (i.e. the heat semigroup) on the torus induces a $\mathcal W_p$-contraction for some $p>1$ is an open question.
\end{remark}

 Considering a lower bound of the local curvature of the drift, namely a function $\eta:\R^d\rightarrow \R$ such that
 \begin{equation}
 \label{eq:eta}
\forall x\in\R^d,\qquad  \eta(x) \geqslant  \sup\{v\cdot\na b(x)v,\ v\in\Sd\}\,,
 \end{equation}
 we immediately get that
 \begin{equation}
 \label{eq:defG}
 \mathbb E_x \po |\Phi_t^X v|^p \pf   \leqslant \mathbb E_{x}\co \exp\po p\int_0^t \eta(X_s) \dd s \pf\cf =: G_p(x,t) \,.
 \end{equation}
This kind of quantities naturally appears in large deviations theory, where $\ln G_p(x,t)$ is referred to as a cumulant. In particular,
\begin{equation}
\label{loc:betakappaG}
\beta_p^p(t) \leqslant \kappa_p^p(t) \leqslant \sup_{x\in \R^d} G_p(x,t) =:  G_p^*(t)\,.
\end{equation}
The objective of the rest of this work is to determine conditions under which the right hand side goes to $0$ as $t\rightarrow \infty$.

In dimension $1$, taking $\eta(x) = b'(x)$,~\eqref{eq:defG} is an equality, and thus $\kappa_p^p(t) = G_p^*(t)$. In higher dimension, it may also be possible to avoid the use of a bound $\eta$ and to study directly $\kappa_p^p(t)$ and get sharp results, by working with a suitable decomposition of the space. This is briefly discussed in Section~\ref{subsec:sharperDecompose} but we won't go into details on this topic, rather focusing in Section~\ref{sec:criteria} on the study of the cumulant $\ln G_p^*$. 

\subsection{Two negative results}\label{subsec:negativeResults}

The classical Bakry-Emery curvature bound gives 
\begin{equation}\label{eq:BEoptimal}
\forall p\geqslant 1,\ \forall t\geqslant 0,\qquad  \beta_p(t) \leqslant \kappa_p(t) \leqslant e^{- \lambda^* t}\,.
 \end{equation}
 The next proposition shows that this bound is in fact optimal for small values of $t$. 
\begin{proposition}\label{prop:tempscourt}
Under Assumption~\ref{Assum:basic}, for all $p\geqslant 1$,
\[\beta_p(t) = 1 -  \lambda^* t + \underset{t\rightarrow 0}o(t)\,.\]
\end{proposition}
This is proven in Section~\ref{subsec:negativeproof}.

The second negative result is that, when considering at a given time $t>0$ the $L^p$ contraction of the synchronous coupling for all $p\geqslant 1$ simultaneously, then it is not possible to get better than the Bakry-Emery curvature. For clarity, let us first state it in the elliptic case:

\begin{proposition}\label{prop:negativeelliptic}
Under Assumption~\ref{Assum:basic}, assume furthermore that $\sigma$ is non-singular. Then, for all $t \geqslant 0$,
\[\lim_{p\rightarrow \infty} \kappa_p(t) = e^{-\lambda^*t}\,.\]
\end{proposition}

This is a particular case of the next proposition, which covers non-elliptic cases. Introduce the set of admissible paths at time $T>0$ as
\[\mathcal A(T) = \left\{\gamma\in\mathcal C([0,T],\R^d):\ \forall \varepsilon>0,\ \mathbb P_{\gamma(0)}\po |X_t-\gamma(t)|\leqslant \varepsilon\ \forall t\in[0,T]\pf >0 \right\}\,.\]
When  $\sigma$ is non-singular, this is simply $\mathcal C([0,T],\R^d)$. Let
\[\overline{\kappa}_\infty(T) = \sup\{|\Phi_T^{\gamma}|,\ \gamma\in\mathcal A(T)\},\]
where $\Phi_t^\gamma $ is the matrix solution to  $\partial_t \Phi_t^\gamma = \na b(\gamma(t)) \Phi_t^\gamma$ with $\Phi_0^\gamma=\mathrm{Id}$. 

\begin{assumption}\label{assum:controllable}
The process~\eqref{eq:EDS_X} is controllable, in the sense that for all $x,y\in\R^d$, and $t,\varepsilon>0$, $\mathbb P_x\po|X_t-y|\leqslant \varepsilon\pf >0$.
\end{assumption}

\begin{proposition}\label{prop:sharppinfini}
Under Assumptions~\ref{Assum:basic} and \ref{assum:controllable}, assume furthermore either one of these two conditions:
\begin{enumerate}
\item $\na b$ is bounded.
\item For any $x,y\in\R^d$, there exists $R>0$ such that for all  $t,\varepsilon>0$,
\[\mathbb P_x\po|X_t-y|\leqslant \varepsilon\, \text{and}\, |X_s|\leqslant R\ \forall s\in[0,t]\pf >0\,.\]
\end{enumerate}
 Then, for any $t\geqslant 0$, $x\in\R^d$ and $v\in\Sd$,
\[\lim_{p\rightarrow \infty} \sup_{v\in\Sd} \co \mathbb E_x\po |\Phi_t^X v|^p\pf\cf^{1/p} \geqslant \overline{\kappa}_\infty(t)\,.\]
In particular, $\lim_{p\rightarrow\infty} \kappa_p(t) \geqslant \overline{\kappa}_\infty(t)$.
\end{proposition}

This is proven in Section~\ref{subsec:negativeproof}. From Proposition~\ref{prop:sharppinfini}, using that Assumption~\ref{assum:controllable} holds when $\sigma$ is non-singular, Proposition~\ref{prop:negativeelliptic} follows by considering 
constant paths $\gamma(s) = x_0$ for $s\in[0,t]$ where $\sup_{v\in \Sd} v\cdot  \na b(x_0) v \geqslant -\lambda^*- \varepsilon$ for an arbitrarily small $\varepsilon>0$, to get that $\overline{\kappa}_\infty(t) \geqslant e^{-\lambda^* t}$.

See Remark~\ref{rem:appliPropkin} for an application of Proposition~\ref{prop:sharppinfini}.

\section{Criteria for long-time contraction} \label{sec:criteria}

This section is devoted to the study of  $G_p^*$ given by~\eqref{loc:betakappaG}, with  $\eta$ satisfying~\eqref{eq:eta}. Thanks to~\eqref{loc:betakappaG}, it is sufficient to show that $G_p^*(t)$  vanishes as $t\rightarrow \infty$ to get the Wasserstein contraction~\eqref{eq:contraction_def} (and even a contraction along synchronous couplings).


\subsection{An asymptotic criterion for contraction}\label{subsec:resultFerreStoltz}
 The long-time behavior of $G_p(x,t)$ for a fixed $x\in\R^d$ is a classical topic of large deviations theory,  related to Varadhan lemma~\cite{burenev2025introduction}. However, since we work in an unbounded space, the supremum over $x$ in~\eqref{loc:betakappaG} has to be carefully controlled. 

We will rely on the general result of Ferré and Stoltz from \cite{FerreStoltz}, which allows both for degenerate diffusion processes and unbounded space. The general conditions from this work are the following:

\begin{assumption}\label{assum:Ferre}
\begin{itemize}
\item The drift $b$ is $\mathcal C^\infty$ with all its derivative growing at most polynomially.
\item The generator $\mathcal L = b\cdot\na + \frac{1}{2}\sigma\sigma^T : \na^2$ satisfies the hypoelliptic Hörmander condition (cf. \cite[Assumption 1]{FerreStoltz}).
\item The process~\eqref{eq:EDS_X} is controllable (cf. \cite[Assumption 2]{FerreStoltz}).
\end{itemize}
\end{assumption}

Finally, \cite{FerreStoltz}  also requires a Lyapunov condition (Assumption 3 there). However, we won't state it since it will be a consequence of the following asymptotic contraction condition, which we will need to deal with the supremum in $x$ in~\eqref{loc:betakappaG}.

\begin{assumption}[Contraction at infinity]\label{Assum:PointwiseContractInfty}
The function $\eta$ satisfies~\eqref{eq:eta} and there exist $R,\rho>0$ such that for all $x\in\R^d,v\in\mathbb S^{d-1}$ with $|x|\geqslant R$, $\eta(x)\leqslant -\rho$ (as a consequence, for any $x\in\R^d,v\in\mathbb S^{d-1}$ with $|x|\geqslant R$, $v\cdot \na b(x) v \leqslant -\rho$). 
\end{assumption}

Building on \cite{FerreStoltz}, our main result in this section is the following:

\begin{theorem}
\label{thm:FerreStoltz}
Under Assumptions~\ref{assum:Ferre} and~\ref{Assum:PointwiseContractInfty}, assume furthermore that $\eta$ is bounded\footnote{This assumption could be weakened to $\eta(x) \ll |x|^2$, cf. \cite{FerreStoltz}} and $\mathcal C^\infty$ with all derivatives growing at most polynomially. Then:
\begin{enumerate}
\item For all $R_0>0$ and $p\geqslant 1$, there exists $C_{R_0}>0$ such that for all $t\geqslant 0$,
\begin{equation}\label{locGCR}
\sup_{x\in \mathcal B(0,R_0)} \left| \frac1t \ln G_p(x,t) - \mathcal J(p\eta)\right| \leqslant \frac{C_{R_0}}{t}\,,
\end{equation}
where, for $f\in L^\infty(\R^d)$, 
\begin{equation}
\label{eq:def-Jpeta}
\mathcal J(f) = \sup_{\nu\in\mathcal P(\R^d)}\left\{\int_{\R^d} f \dd \nu - \mathcal I(\nu)\right\}\,,
\end{equation}
with the rate function 
\[\mathcal I(\nu) = \sup_{u\in \mathcal D_+(\mathcal L)}\left\{ -\int_{\R^d} \frac{\mathcal L u}{u} \dd \nu \right\}\,,\]
 where, denoting by $\mathcal D(\mathcal L)$ the extended domain of the generator $\mathcal L$ (cf. \cite[Section 3.1]{FerreStoltz}) and $\mathcal C(\R^d)$ the space of continuous functions,
 \[\mathcal D_+(\mathcal L) = \left\{ u \in \mathcal D(\mathcal L) \cap \mathcal C(\R^d), \ u>0,\ \frac{\mathcal L u}{u} \in L^\infty(\R^d)\right\}\,.\]
\item For all $p\geqslant 1$, there exists $C>0$ such that for all $t\geqslant 0$,
\begin{equation}
\label{eq:limsupGJpeta}
  G_p^*(t)  \leqslant C e^{ t \max\po -\rho, \mathcal J(p\eta)\pf}  \,.
\end{equation}
If, moreover, there exists $K>0$ such that $|b(x)|\leqslant K|x|+K$ for all $x\in\R^d$, then for all $p\geqslant 1$, there exists $c>0$ such that for all $t\geqslant 0$,
\begin{equation}
\label{e:liminfGJpeta}
   G_p^*(t) \geqslant c  e^{ t \max\po -\rho', \mathcal J(p\eta)\pf}\,,
\end{equation}
with $\rho' = -\liminf_{|x|\rightarrow \infty}\eta(x) \in[\rho,\infty]$.
\end{enumerate}
\end{theorem}

Moreover, as explained in \cite[Remark 9]{FerreStoltz}, $\mathcal J(f)$ is the largest eigenvalue of the Feynman-Kac operator $L_f=\mathcal L + f$ (see Section~\ref{sec:numerique} on this topic).

The proof of Theorem~\ref{thm:FerreStoltz} is provided in Section~\ref{sec:preuveThemFerreStoltz}.

\begin{remark}
The relation between a Wasserstein contraction and the largest eigenvalue of $L_{p\eta}$ is already discussed in \cite[Section 5]{MonmarcheBruit}. However, the direction sketched there, using a weighted distance, would involved a rougher lower-bound on the curvature ($k$ there instead of $-\eta$ here) and would require the study of the associated eigenvector. This eigenvector is considered and exploited in \cite{FerreStoltz} (whose result we are using), which means that we are indirectly also relying on this study, however we only use it for the control~\eqref{locGCR} on compact sets, which means that we do not have to  understand the behavior of the eigenvector as $|x|\rightarrow \infty$. 
\end{remark}

\begin{remark}\label{rem:Jcaselliptique}
A clearer expression for $\mathcal I(\nu)$ is given in \cite[Corollary 3]{FerreStoltz}. Assume that $\nu = e^{h} \mu_\infty$ with $\mu_\infty$ the invariant measure of $X$ and decompose the generator $\mathcal L = \mathcal L_S + \mathcal L_A$ with, denoting by $\mathcal L^*$ the adjoint operator of $\mathcal L$ in $L^2(\mu_\infty)$, $\mathcal L_S=(\mathcal L+\mathcal L^*)/2$ and $\mathcal L_A = (\mathcal L - \mathcal L^*)/2$. Let $S=\sigma\sigma^T/2$, so that the carré du champs associated to $\mathcal L$, defined by $2\Gamma(f) = \mathcal L(f^2) - f\mathcal L f$, is simply $\na f\cdot S\na f$. Then
\[\mathcal I(\nu) = \frac14 \int_{\R^d} \co \Gamma(h) + \Gamma(\psi_h)\cf \dd \nu \]
where, denoting by $\tilde \na g = -\na g - g \ln \nu $ the adjoint of $\na$ in $L^2(\nu)$, $\psi_h$ is defined as the solution of the Poisson equation
\[\tilde\na \po S \na \psi_h\pf = \mathcal L_A h\,.\]
\end{remark}

\begin{example}\label{ex:Joverdamped}
As an application of Remark~\eqref{rem:Jcaselliptique}, for the overdamped Langevin process
\begin{equation}
\label{eq:overdampedSDE}
\dd X_t = -\na U(X_t) \dd t + \sqrt{2}\theta\dd B_t
\end{equation}
with $\theta>0$, we get the Fisher Information
\[\mathcal I(\nu) = \frac{\theta^2}{4}\int_{\R^d} \left|\na \ln \frac{\nu}{\mu}\right|^2\dd\nu \,.\]
In particular, if $|\na \ln \mu_\infty|^2 = |\na U|^2$ grows polynomially for instance, then $\mathcal I(\nu) < \infty$ for any Gaussian density. By taking $\nu = \mathcal N(x_0,\varepsilon I_d)$ with an $x_0\in\R^d$ such that $\eta(x_0)$ is arbitrarily close to $\sup \eta$ and by taking $\varepsilon$ arbitrarily small, we can have $\int_{\R^d} \eta \nu$ arbitrarily close to $\sup\eta$ while having $\mathcal I(\nu) <\infty$.  From~\eqref{eq:def-Jpeta}, we deduce that $\mathcal J(p\eta)/p \rightarrow \sup\eta$ as $p\rightarrow\infty$.
\end{example}

\begin{remark}
In fact, in the elliptic case~\eqref{eq:overdampedSDE}, a consequence of \cite[Corollary 4]{wu2000deviation} is that, provided the invariant measure $\mu_\infty\propto e^{-U/\theta^2}$ satisfies a log-Sobolev inequality:
\begin{equation}
\label{eq:LSI}
\exists C_{LS}>0,\, \forall \nu \ll \mu_\infty,\qquad \int_{\R^d} \ln \po \frac{\dd\nu}{\dd \mu_\infty}\pf \dd \nu \leqslant \frac{\theta^2 C_{LS}}{4} \int_{\R^d} \left|\na \ln \po \frac{\dd\nu}{\dd \mu_\infty}\pf\right|^2 \dd \nu\,,
\end{equation}
 then for all $t\geqslant 0$ and $p\geqslant 1$, 
 \[\frac1t \ln \int_{\R^d} G_p(x,t) \mu_\infty(\dd x) \leqslant \frac{1}{C_{LS}} \ln \int_{\R^d} e^{C_{LS} p\eta(x)}\mu_\infty(\dd x)\,.\] 
 Lower-bounding the left-hand side by an integral over a compact set and using~\eqref{locGCR} shows that
 \[\mathcal J(p\eta) \leqslant \frac{1}{C_{LS}} \ln \int_{\R^d} e^{C_{LS} p\eta(x)}\mu_\infty(\dd x)\,. \] 
 Alternatively, using Pinsker's inequality and then~\eqref{eq:LSI}, for any $\nu\in\mathcal P(\R^d)$, denoting by $\mathrm{osc}(\eta)= \sup \eta - \inf \eta$ the oscillation of $\eta$ and by $\|\cdot\|_{TV}$ the total variation norm,
 \begin{align*}
\int_{\R^d} p\eta\dd\nu &\leqslant  \int_{\R^d} p \eta \dd \mu_\infty + p\mathrm{osc}(\eta) \|\nu-\mu_\infty\|_{TV} \\
& \leqslant \int_{\R^d} p \eta \dd \mu_\infty + p\mathrm{osc}(\eta) \sqrt{ C_{LS} \mathcal I(\nu)} \,,
 \end{align*}
 from which
 \[\mathcal J(p\eta) \leqslant \int_{\R^d} p \eta \dd \mu_\infty + C_{LS}(p\mathrm{osc}(\eta))^2 /2\,.\]
 With the same reasoning but applying the Talagrand's inequality implied by the log-Sobolev inequality~\cite{otto2000generalization} instead of Pinsker's gives
 \[\mathcal J(p\eta) \leqslant \int_{\R^d} p \eta \dd \mu_\infty + (C_{LS}p L_\eta)^2 \,,\]
 with $L_\eta$ the Lipschitz constant of $\eta$.
\end{remark}

\begin{remark}\label{rem:LSI}
Conversely, in the general case, notice that if $\mathcal J(2\eta) <0$ and Assumption~\ref{Assum:PointwiseContractInfty} holds, then the $L^2$ Wasserstein contraction provided by Theorem~\ref{thm:FerreStoltz}, through its equivalence with the gradient estimate~\eqref{eq:naPt2}, implies a Poincaré inequality for $\mu_\infty$ (see the proof of \cite{MonmarcheBruit}, which requires neither reversibility nor ellipticity). Moreover, Assumption~\ref{Assum:PointwiseContractInfty} implies a so-called defective log-Sobolev inequality (see \cite[Proposition 4]{monmarche2025logarithmic}), which together with the Poincaré inequality implies that $\mu_\infty$ satisfies a log-Sobolev inequality (associated with the classical Fisher Information, and not the rate function $\mathcal I$, and neither the Dirichlet form of the diffusion process, which may be degenerate when $\sigma$ is singular).
\end{remark}

In view of \eqref{loc:betakappaG}, the consequence of Theorem~\ref{thm:FerreStoltz} (specifically,~\eqref{eq:limsupGJpeta}) for our purpose is that, under Assumptions~\ref{assum:Ferre} and~\ref{Assum:PointwiseContractInfty}, a $\mathcal W_p$ contraction holds when $\mathcal J(p\eta) <0$. 

It is often not easy to prove that $\mathcal J(p\eta)$ is negative, but it can be estimated numerically, as detailed in Section~\ref{sec:numerique}. 

In particular, in dimension $1$ for instance, taking $\eta= b'$, then $\kappa_p^p(t) = G_p^*(t)$ and the lower bound~\eqref{e:liminfGJpeta} together with Proposition~\ref{prop:COntractionkappap} shows that synchronous coupling does not induce an $L^p$-contraction when $\mathcal J(p\eta) \geqslant  0$ (in higher dimension, obtaining the same conclusion would require to go into the details of the discussion in Section~\ref{subsec:sharperDecompose}). As seen in Example~\ref{ex:Joverdamped}, for the overdamped Langevin diffusion~\eqref{eq:overdampedSDE}, as soon as $U$ is not convex, then we obtain (in dimension 1) that the synchronous coupling does not induce an $L^p$-contraction for $p$ large enough. Besides, if $U$ is convex but not strongly convex, the contraction rate vanishes as $p\rightarrow \infty$. These remarks are consistent with Proposition~\ref{prop:negativeelliptic}.

\medskip

As mentioned, Theorem~\ref{thm:FerreStoltz} is based on \cite{FerreStoltz}, which crucially relies on a long-time convergence estimate for a Feynman-Kac operator (see in particular Equation (89) in this reference), where the constants are not explicit. This is due to the use of a compactness argument, together with hypoellipticity and controllability, to prove a minorization condition in order to apply Harris theorem \cite{HairerMattingly2008}. In some situations (elliptic or kinetic diffusions for instance), it could be possible to use a reflection coupling or a Wang-Harnack inequality to avoid this compactness argument, and then to go through the proofs in~\cite{FerreStoltz} to get explicit constants in the final result, hence getting an explicit non-asymptotic bound in Theorem~\ref{thm:FerreStoltz} instead of~\eqref{eq:limsupGJpeta} (with an unknown $C$). However this is beyond the scope of the present work. Moreover, as long as the argument relies on the Harris theorem in terms of total variation, it is unclear whether such constructive estimates would have nice dependencies in the dimension $d$, which is often the main parameter of interest (since total variation behaves badly with dimension).

Instead, as an alternative to Theorem~\ref{thm:FerreStoltz}, under more restrictive conditions, in  Section~\ref{subsec:non-asymp} we will give a non-asymptotic bound on $ G_p^*(t)$ (hence $\beta_p(t)$), independent from $d$ (provided all parameters in the assumptions are independent from $d$).

\subsection{Numerical experiments}\label{sec:numerique}



As mentioned above, $\mathcal J(p\eta)$  is characterized as the largest eigenvalue of the Feynman-Kac operator $L_{p\eta}$ given by
\begin{equation}
\label{eq:FKeigen}
L_{p\eta} f(x) = b(x) \cdot \na f(x) + \frac12 \sigma^2 : \na^2 f(x) + p\eta(x) f(x)\,, 
\end{equation}
see \cite{burenev2025introduction,FerreStoltz}. In dimension 1, this can be easily solved numerically. In Figure~\ref{fig:Lyap} we display the result when considering the process
\begin{equation}
\label{locXU}
\dd X_t = - U'(X_t) \dd t + \sqrt{2}\theta \dd B_t
\end{equation}
with multi-well potentials $U\in\{U_1,U_2\}$ given by $U_i(x)=x^2 + 2e^{-x^2} + a_i \cos(10x) $, $i\in\{1,2\}$, $(a_1,a_2)=(0,0.25)$, with $\eta(x)=-U''(x)$.  The eigenvalue problem associated to~\eqref{eq:FKeigen} is truncated over $x\in [-5,5]$ and discretized with a step-size $\Delta x= 10^{-3}$, with 25 values evenly spaced for each parameter $p\in[1,3]$, $\theta^2 \in[0.1,5]$. Tested with $U_0(x)=x^2$ (for which $\mathcal J(p\eta)/p = -2$) these discretization parameters gave an accuracy of order $10^{-6}$ over this range of $p,\theta^2$. Let us make the following comments on Figure~\ref{fig:Lyap}:
\begin{itemize}
\item These potential $U_i$ being strongly convex outside a compact set,  a Wasserstein contraction~\eqref{eq:contraction_def} for $p=1$ follows from~\cite{Eberle1}. Due to Proposition~\ref{prop:dim1} (and since in dimension $1$ here~\eqref{eq:defG} is in fact an inequality), this implies that $\mathcal J(\eta)<0$ for all $\theta$. This is indeed observed in Figure~\ref{fig:Lyap}.
\item  Similarly, the result of \cite{MonmarcheBruit} applies in the present situation, which means (thanks to Proposition~\ref{prop:COntractionkappap})  that  for any $p\geqslant 1$, for a temperature $\theta^2$ large enough, $\kappa_p(t) <1$ for $t$ large enough. This implies that, for any $p$, $\mathcal J(\eta p) <0$ for  $\theta^2$ large enough, as observed in Figure~\ref{fig:Lyap}.
\item Conversely, it appears in Figure~\ref{fig:Lyap} that for any $p>1$, $\mathcal J(p\eta)>0$ when $\theta^2$ is small. Since the numerical approximation performed in the simulation have theoretical convergence guarantees, these simulations could be used to prove rigorously that indeed the real value of $\mathcal J(p\eta)$ for a given $p$ (e.g. $p=2$) is negative for small values of $\theta^2$, although we won't detail this. This shows that the restriction in~\cite{MonmarcheBruit} (which is based on synchronous couplings) to high temperatures is in fact necessary. Moreover, the critical temperature at which $\mathcal J(p\eta)$ changes sign appears to depend on $p$, in consistency with the results in~\cite{MonmarcheBruit}.
\item For any $\theta^2>0$, we see that $\mathcal J(\eta p) >0$ for $p$ large enough, which is a consequence of Example~\ref{ex:Joverdamped}.
\end{itemize}

A natural question in view of these observations is whether there is no $\mathcal W_2$-contraction at low temperature, namely: for~\eqref{locXU} with a non-convex $U$, do we necessarily have that $\beta_2(t) \geqslant 1$ for all $t\geqslant 0$ when $\theta^2$ is small enough ? We shall not answer it in this work and leave it as an open problem.

\begin{figure}
\begin{center}
\includegraphics[scale=0.35]{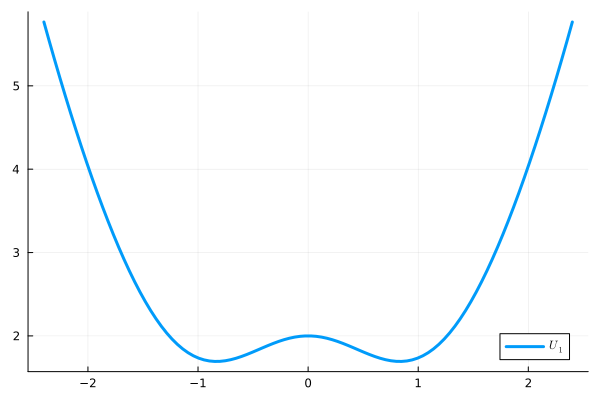}
\includegraphics[scale=0.35]{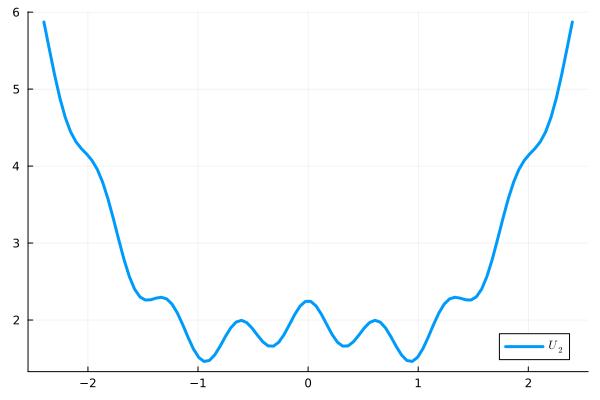}

\includegraphics[scale=0.35]{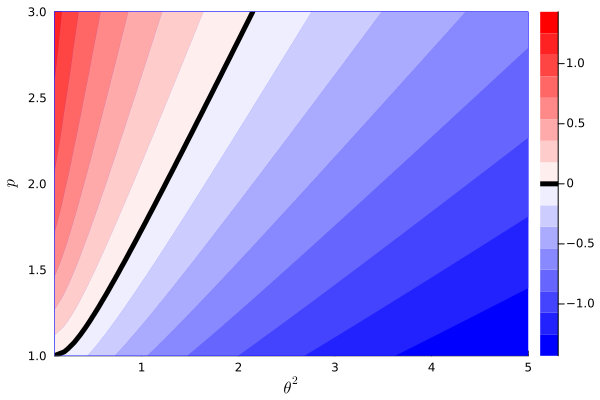}
\includegraphics[scale=0.35]{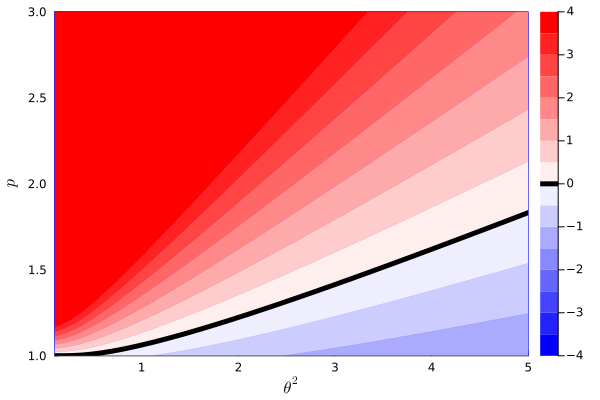}
\caption{Top: potentials $U_1$ (left) and $U_2$ (right). Bottom: estimation of $\mathcal J(p\eta)/p$ for~\eqref{locXU} (left with $U_1$, right with $U_2$), as a function of $\theta^2\in[0.1,2]$ and $p\in[1,3]$ (on the right graph, the color bar is capped in $[-4,4]$).}\label{fig:Lyap}
\end{center}
\end{figure}

\subsection{A constructive non-asymptotic criterion}\label{subsec:non-asymp}

The condition that $\mathcal J(p\eta)<0$ to get a $\mathcal W_p$-contraction thanks to Theorem~\ref{thm:FerreStoltz} is conceptually interesting but difficult to check in practice. In this section, we prove the contraction under more restrictive but explicit conditions.

The first maint step, of interest for itself, is  a general result of Wasserstein long-time contraction for the cost $\omega(x,x')=\max(|x-x'|,|x-x'|^p)$, using explicit Markovian coupling as in e.g. \cite{Eberle1,EberleGuillinZimmer,HerdaPerthameMonmarche} (this is Theorem~\ref{thm:contractionMarkov}) below). In particular, for $p=1$, this gives a $\mathcal W_1$ contraction, which is known to induce concentration inequalities (see \cite{Djellout,CattiauxGuillinFathi} and Section~\ref{subsec:concentration}) and thus provides a control on $G_p^*(t)$, eventually leading to a $\mathcal W_p$-contraction possibly for $p>1$.

Concerning the contraction with the cost $\omega$, the case of elliptic diffusion processes is already covered by \cite{Eberle1} and the case where the drift induces a global contraction is well-known \cite{MonmarcheContraction}. As a consequence, we focus on a general hybrid situation where the state can be decomposed as $x=(y,z)$, with an elliptic noise on  the coordinates $z$  while, considering two synchronously-coupled processes, the coordinates $y$ are contracted by the drift $b$ when the coordinates $z$ are the same (or more generally, when the distance between the $z$ variables is small with respect with the distance between the $y$ variables). This is stated as follows:

\begin{assumption}[State decomposition]\label{assum:decomposeYZ}
The state is decomposed as $x = (y,z)\in\R^{n+m}$ for some $n,m\in\N^*$ and there exist $\rho_1,L_1,L_2,L_3,\theta>0$ such that for all $x=(y,z),x'=(y',z')\in\R^{n+m}$, writing $b(x)=(b_1(x),b_2(x)) \in \R^{n}\times \R^m$,
\begin{equation}
\begin{array}{rcl}
(b_1(x) - b_1(x')) \cdot (y-y') & \leqslant - \rho_1 |y-y'|^2 + L_1 |y-y'||z-z'| \\
(b_2(x)-b_2(x')) \cdot (z-z') & \leqslant L_2|y-y'||z-z'| + L_3|z-z'|^2 \,,
\end{array}\label{locAssumYZ}
\end{equation}
and $\Sigma = \sigma\sigma^T$ is decomposed by blocks as
\begin{equation}
\label{eq:decomposSigma}
\Sigma = \theta^2 \begin{pmatrix}
0 & 0\\
0  & \mathrm{Id}_m
\end{pmatrix} + \tilde \sigma \tilde \sigma^T\,,\qquad \tilde \sigma = \begin{pmatrix}
\tilde \sigma_1 \\ \tilde \sigma_2
\end{pmatrix}  \in \R^{(n+m)^2}\,.
\end{equation}
\end{assumption}

\begin{remark}\label{rem:n0}
It is possible to consider the case where $n=0$; in that case, we are simply back to the settings of~\cite{Eberle1}.
\end{remark}

Such a decomposition already appears in the kinetic Langevin and Morris-Lecar processes considered in~\cite{EberleGuillinZimmer,HerdaPerthameMonmarche}, and our strategy and results will be similar to these references. By comparison with these earlier works (and other alternative studies on the kinetic Langevin case \cite{schuh2024global,cheng2018sharp}), our main objective is to have a flexible result with conditions which are easily checked in practice on various models, as illustrated in Section~\ref{sec:applications}. Moreover, the decomposition~\eqref{eq:decomposSigma} is interesting in the high-diffusivity regime, for instance if  
\[\Sigma = \theta_0^2 \begin{pmatrix}
0 & 0\\
0  & \mathrm{Id}_m
\end{pmatrix}\]
for some $\theta_0>0$ then we may choose any $\theta \in(0,\theta_0]$ in~\eqref{eq:decomposSigma}. Since our results will not depend on the matrix $\tilde \sigma$, we will get convergence rates which are uniform in $\theta_0 \in [\theta,\infty)$ for any $\theta>0$. Similarly, if we consider a sequence of processes where the diffusion matrix $\sigma_n $ converges as $n\rightarrow \infty$ to a degenerate $\sigma$ (for instance, $\sigma_n = \sigma + \frac1n I_d$), then we can have a decomposition~\eqref{eq:decomposSigma} with $\theta^2$ independent from $n$.

Assumption~\ref{assum:decomposeYZ}  is clearly not sufficient by itself to ensure the stability of the process, let alone a Wasserstein contraction. We complement it with the following:

\begin{assumption}[Contractivity at infinity]
\label{assum:Contract_at_infty}
There exist $\rho_2,S^*>0$ and a positive definite symmetric matrix $Q \in \R^{d\times d}$ such that
\[(x-x') \cdot Q \po b(x)-b(x')\pf \leqslant -  \rho_2 \|x-x'\|_Q^2\qquad \forall x,x'\in\R^d,\ \|x-x'\|_Q\geqslant S^*\,,\]
where $\|x\|_Q = \sqrt{x\cdot Q x}$.
\end{assumption}

For a given $Q$, this condition is equivalent to the same condition with $Q=\mathrm{Id}$ but for the process $Q^{1/2} X_t$ instead of $X_t$. The interest is that we can work with different variables to check Assumptions~\ref{assum:decomposeYZ} and \ref{assum:Contract_at_infty}.

Under both Assumptions~\ref{assum:decomposeYZ} and \ref{assum:Contract_at_infty}, we write $\mathfrak{P}'=\{\rho_1,L_1,L_2,L_3,\theta,|Q|,|Q^{-1}|,\rho_2,S^*\}$ the relevant parameters. Again, we emphasize that neither the dimensions $n,m$ nor the matrix $\tilde \sigma$ appear in $\mathfrak{P}'$.

\begin{theorem}
\label{thm:contractionMarkov}
Under Assumptions~\ref{assum:decomposeYZ} and \ref{assum:Contract_at_infty}, for any $p\geqslant 1$, there exist $C_p,\lambda_p>0$ (which depend only and explicitly on $p$ and the parameters in $\mathfrak{P}'$) such that for any $\nu,\mu\in\mathcal P_p(\R^{n+m})$ and $t\geqslant 0$,
\begin{equation}
\label{eq:weightWp}
\mathcal W_{\omega}(\nu P_t,\mu P_t) \leqslant C_p e^{-\lambda_p t}  \mathcal W_{\omega}(\nu ,\mu)\,,
\end{equation}
with $\omega(x,x')=\max(|x-x'|,|x-x'|^p)$. 
\end{theorem}

The proof is given in Section~\ref{sec:MarkovianCoupling} (with explicit expressions for $C_p$ and $\lambda_p$ provided at the end of the proof).

Applying this result with $p=1$ gives the Wasserstein contraction~\eqref{eq:contraction_def} in this case. By the Banach fixed-point theorem, the semi-group property and conditionning with respect to the initial condition, this classically implies  that $(P_t)_{t\geqslant 0}$ admits a unique invariant measure $\mu_\infty$ and that $\nu P_t$ converges weakly to $\mu_\infty$ for all $\nu \in \mathcal P(\R^d)$. Moreover, as mentioned above and further details in Section~\ref{subsec:concentration}), it induces concentration inequalities, from which we deduce the following:

\begin{proposition}\label{prop:concentreGp*}
Under Assumptions~\ref{Assum:basic} and~\ref{Assum:PointwiseContractInfty}, assume moreover that $\eta$ is Lipschitz continuous with constant $L_\eta$ and that the contraction~\eqref{eq:contraction_def} holds with $p=1$ for some $C_1,\lambda_1>0$. Denote by $\mu_\infty$ the invariant measure of the diffusion process~\eqref{eq:EDS_X}. Then, for all $p\geqslant 1$ and $t\geqslant 0$,
\[G_p^*(t) \leqslant A e^{- \min(\rho,\rho') t }\]
 with
 \begin{equation}
 \label{locArho'}
 A= \exp\left(     \frac{p L_\eta C_1}{\lambda_1} \po R+\int_{\R^d}|y|\mu_\infty(\dd y)\pf  \right)\,, \qquad		\rho' =  - p \mu_\infty(\eta) -  \frac{  |\sigma|^2 C_1^2 p^2 L_\eta^2 }{\lambda_1^2} \,.
 \end{equation}
\end{proposition}
This is proven in Section~\ref{subsec:concentration}. It only gives a contraction when $\rho'>0$, which a minima requires that $\mu_\infty(\eta)<0$. As discussed in \cite{CattiauxGuillinFathi}, this condition is met in many situations, for instance for the overdamped Langevin diffusion~\eqref{eq:overdampedSDE} in dimension $d=1$, with $\theta=1$ for simplicity (so that $\mu_\infty \propto e^{-U}$), with $\eta(x) = -U''(x)$, integrating by parts yields
\[\mu_\infty(\eta) = - \int_{\R^d} (U')^2 \mu_\infty < 0\,. \]
However, even under this condition, $\rho' >0$ only for sufficiently small values of $p$. For a given $p\geqslant 1$, the question is to find a $\eta$ satisfying~\eqref{eq:eta} with $\mu_\infty(\eta)<0$ but with a sufficiently small Lipschitz constant so that $\rho'>0$. Relying on Theorem~\ref{thm:contractionMarkov} to control $C_1$ and $\lambda_1$, we can get the following general result:

\begin{theorem}\label{thm:final}
For $\unif{\rho},\unif{\rho}',\unif{S},\unif{M},\unif{\theta},\unif{q}>0$, $\unif{S_2}>\unif{S}$ and $p\geqslant 1$, there exist $\unif{\delta},\unif{C},\unif{\lambda}>0$ (depending only and explicitly on $\unif{\rho},\unif{\rho}',\unif{S},\unif{S_2},\unif{M},\unif{\theta},\unif{q},p$) such that the following holds. Let  $d\in\N^*$, $b\in\mathcal C^1(\R^d,\R^d)$ and $\sigma\in\R^{d\times d}$ satisfying Assumptions~\ref{assum:decomposeYZ} and \ref{assum:Contract_at_infty} with
\begin{equation}
\label{locboundparam}
\min(\rho_1,\rho_2) \geqslant \unif{\rho}\,,\quad \max(L_1,L_2,L_3,|Q|,|Q^{-1}|,|\sigma|) \leqslant \unif{M}\,,\quad S^* \leqslant \unif{S},\quad \theta \geqslant \unif{\theta},
\end{equation}
and such that, moreover, for all $x\in\R^d$ and $v\in\R^d$ with $v\cdot Q v= 1$,
\begin{equation}
\label{locvQv}
v\cdot Q \na b(x) v \leqslant \left\{\begin{array}{ll}
\unif{\delta}&  \text{if }|x| < \unif{S} \\
- \unif{\rho} & \text{if }|x|\geqslant\unif{S}\,,
\end{array}\right.
\end{equation}
and  the invariant measure\footnote{whose existence and uniqueness is implied by Theorem~\ref{thm:contractionMarkov} applied with $p=1$.} $\mu_\infty$ of~\eqref{eq:EDS_X} satisfies
\begin{equation}
\label{eq:muinftyq}
\mu_\infty\po \mathcal B(0,\unif{S_2})\pf \leqslant 1 - \unif{q}\,,
\end{equation}
where $\mathcal B(x,r)=\{y\in\R^d,\ |y-x|\leqslant r\}$.  
Then, the $L^p$ Wasserstein contraction~\eqref{eq:contraction_def} holds with $C = \unif{C}$ and $\lambda = \unif{\lambda}$.
\end{theorem}

The proof is provided in Section~\ref{subsec:demofinal}. Notice that it is not difficult to see that~\eqref{locvQv} implies Asumption~\ref{assum:Contract_at_infty} (as in \cite[Section 2.2.1]{MonmarcheBruit}). 

 Keeping track of the explicit constants in Theorem~\ref{thm:final} is tedious and not very informative, however the important point is  that there is no dependency in $d$ (when considering that the parameters in Assumptions~\ref{assum:decomposeYZ} and \ref{assum:Contract_at_infty}  are independent from $d$). The novelty with respect to the classical results in the globally dissipative case is that in~\eqref{locvQv} we allow the curvature to be negative in a compact set (although it cannot be too negative depending on the other parameters).

Notice that, in the elliptic case, this doesn't recover the result of~\cite{MonmarcheBruit}. In the latter, there is no restriction on how negative  the curvature of the drift can be in a ball, but then the diffusivity $\theta$ has to be large enough. This is not an equivalent condition, and in our situation we were not able to play only with the temperature to conclude (in particular, in~\eqref{locArho'}, it is not clear that increasing $\theta$ helps to get $\rho'>0$, even when taking into account that $\mu_\infty$ gets more spread out and thus $\eta$ can have a small Lipschitz constant while $\mu_\infty(\eta)$ remains negative).

It would be interesting to have a result similar to Theorem~\ref{thm:final} (in particular, covering non-elliptic processes) but with a condition only on the temperature as in~\cite{MonmarcheBruit}. Indeed, observing a change of behaviour (i.e. a phase transition) as the temperature varies is a natural question in statistical physics, and moreover many enhanced sampling algorithms are based on the general idea to sample a measure at a higher temperature than the real target, see  \cite{lelievre2025convergence} and references within.

\begin{remark}\label{rem:q}
When $\mu_\infty$ is only known as the invariant measure of~\eqref{eq:EDS_X} (without for instance any explicit density or other informations), a general way to obtain an inequality of the form~\eqref{eq:muinftyq} is to find a function $W \in \mathcal C^2(\R^d, \R)$ satisfying
\begin{equation}
\label{eq:remqW}
\mathcal L W(x) \leqslant \left\{\begin{array}{ll}
-\varepsilon & \text{if }x\in \mathcal D\\
C & \text{if }x\notin \mathcal D
\end{array} \right.
\end{equation}
for some domain $\mathcal D \subset \R^d$ where $C,\varepsilon>0$ and we recall that  $\mathcal L = b\cdot\na + \frac{1}{2}\sigma\sigma^T : \na^2$ is the generator of the process. Indeed, using that $\int_{\R^d} \mathcal L W \dd \mu_\infty = 0$ by invariance, this implies that
\[\mu_\infty \po \mathcal D \pf \leqslant \frac{C}{C+\varepsilon}\,.\]

\end{remark}

\section{Examples}\label{sec:applications}

\subsection{The elliptic case}

We have already discussed the reversible overdamped Langevin process in Example~\ref{ex:Joverdamped}. When the drift $b$ is not the gradient of a potential and $\sigma=\theta \mathrm{Id}_d$ for some $\theta>0$, the process is not reversible but the results of~\cite{MonmarcheBruit} apply and give a clearer result that Theorem~\ref{thm:final} (with simple explicit constants). However, for illustrative purpose, let us show how Theorem~\ref{thm:final} applies in this situation:

\begin{proposition}\label{prop:elliptic}
Given $\theta,L,\ell,R,K,R_2>0$  and $p\geqslant 1$, there exist $C_p,\lambda_p,\delta>0$ such that for any $d\in\mathbb N^*$, $b\in\mathcal C^2(\R^d,\R^d)$ such that for all $v\in\Sd$,
\[\forall x\in\R^d\,,\qquad      v\cdot \na b(x) v  \geqslant \ell \1_{|x|\geqslant R} - \delta \1_{|x|<R}    \,,\]
and
\[\forall x \in \mathcal B(0,R_2)\,,\qquad x\cdot b(x) \geqslant - K |x|^2\,, \]
the $L^p$ contraction~\eqref{eq:contraction_def} holds with $C=C_p$ and $\lambda = \lambda_p$ for the process~\eqref{eq:EDS_X} with $\sigma = \theta \mathrm{Id}_d$.
\end{proposition}

\begin{proof}
Apart from~\eqref{eq:muinftyq}, all the others conditions of Theorem~\ref{thm:final} are straightforward to check here (with $Q=\mathrm{Id}$, and Assumption~\ref{assum:decomposeYZ} to be understood with $n=0$, see Remark~\ref{rem:n0}). In particular,  the local inequality $ v\cdot \na b(x) v  \geqslant \ell$ for all $v\in\Sd$ when $|x|\geqslant R$ implies Assumption~\ref{assum:Contract_at_infty} by integrating (see e.g. \cite[Section 2.2.1]{MonmarcheBruit}, or the more complicated situation in the kinetic case in Section~\ref{subsec:kinetic}).

 To apply Remark~\ref{rem:q}, take $W(x) = g(|x|^2/2)$ where $g(0)=0$, $g'(0)=-1$ and
\[g''(s) = \left\{\begin{array}{ll}
\displaystyle{- \frac{2K}{\theta^2} \exp\po \frac{2K s}{\theta^2}\pf }& \text{for } s\in [0,R^2/2]\\
\displaystyle{\frac{2}{R_2^2-R^2} \exp\po \frac{K R^2}{\theta^2}\pf}  & \text{for } s \in ]R^2/2,R_2^2/2[\\
\displaystyle{0} & \text{for }s \geqslant R_2^2/2\,,
\end{array}\right. \]
which gives $g'(s) = -\exp\po \frac{2K s}{\theta^2}\pf $ for $s\in[0,R^2/2]$ and $g'(s) = 0$ for $s\geqslant R_2^2/2$. This doesn't make $W$ a $\mathcal C^2$ function, so in fact what we really do is take $g''$ as a continuous function which is equal to this function over $[0,R^2/2]\cup[R^2/2+\varepsilon,R_2^2/2-\varepsilon]\cup [R_2^2/2,\infty)$ for some small $\varepsilon>0$ and such that $g'(s)=0$ for all $s\geqslant R_2^2/2$ and $g' \leqslant 0$ everywhere, do the computation with this $g$ and let $\varepsilon$ vanish at the end. We omit this step and do the computations with the $g$ above, which eventually provides the same result. Writing $s=|x|^2/2$,
\begin{align*}
\mathcal LW(x) &= b(x)\cdot x g'(s) + \frac{\theta^2}2 \co |x|^2  g''(s) + d g'(s)\cf\\
&\leqslant  - K |x|^2 g'(s) +  \frac{\theta^2}2 \co |x|^2  g''(s) + d g'(s)\cf\\ 
&\leqslant \left\{\begin{array}{ll}
\frac{\theta^2 d}2  g'(s) & \text{if } |x|< R\\
K R_2^2\|g'\|_\infty +  \frac{\theta^2 R_2^2}{2(R_2^2-R^2)} \exp\po \frac{K R^2}{\theta^2}\pf & \text{if }x>R
\end{array}\right.
\end{align*}
Hence, we get~\eqref{eq:remqW} with
\[C= \co K+\frac{\theta^2 }{2(R_2^2-R^2)} \cf R_2^2  \exp\po \frac{K R^2}{\theta^2}\pf \qquad \varepsilon =  \frac{\theta^2 d}2\,, \]
and thus, thanks to Remark~\ref{rem:q}, the condition~\eqref{eq:muinftyq} with
\[\overline{q} =  \frac{\varepsilon}{C+\varepsilon}\,. \]
This concludes the proof, since all conditions of Theorem~\ref{thm:final} have been checked.
\end{proof}

\subsection{Kinetic Langevin process}\label{subsec:kinetic}

Consider the kinetic Langevin diffusion process $(q_t,p_t)\in\R^{2d}$, solution of
\begin{equation}
\label{eq:kinLangevin}
 \left\{\begin{array}{lcl}
\dd q_t &=& p_t \dd t\\
\dd p_t &= & F(q_t,p_t) \dd t + \sqrt{2}\theta \dd B_t\,,
 \end{array}\right.
\end{equation}
with a force $F\in \mathcal C^1(\R^d)$ with bounded derivative and $\theta>0$. In the classical equilibrium case, $F(q,p)=-\na V(q) - \gamma p$ for some potential $V$ and friction $\gamma>0$.

 This process satisfies Assumption~\ref{assum:Ferre}, cf. \cite{FerreStoltz}. In order to apply either Theorem~\ref{thm:FerreStoltz} or \ref{thm:final}, we need a function $\eta$ satisfying~\eqref{eq:eta} and negative outside a compact set (we don't need it to apply Theorem~\ref{thm:contractionMarkov}, but then we are just recovering results similar to~\cite{EberleGuillinZimmer,schuh2024global}). However, in the equilibrium case for instance, for any $x=(q,p)$, for the drift $b(x) = (p,-\na V(q) -\gamma p)$,
\begin{equation}
\label{locnab}
\na b(x) = \begin{pmatrix}
0 & \mathrm{Id} \\
- \na^2 V(q) & -\gamma \mathrm{Id}
\end{pmatrix}\,.
\end{equation}
In particular, taking $v=(1,0,\dots,0)$ gives $v\cdot \na b(x) v = 0$. It is thus necessary to work with different coordinates. Following \cite{MonmarcheContraction}, we are led to work with $z = M^{1/2}x$
with
\begin{equation}
\label{locM}
M =  \begin{pmatrix}
\mathrm{Id} & c \mathrm{Id}\\
c \mathrm{Id} & a\mathrm{Id}
\end{pmatrix} \,,
\end{equation}
with $a,c\in\R$ to be chosen with  $a>c^2$. Having $\eta$ satisfying~\eqref{eq:eta} with the drift of the process $Z_t=M^{1/2}(q_t,p_t)$ amounts to enforce
\begin{equation}
\label{locetaM}
 \eta(z) \geqslant v \cdot M \na b(M^{-1/2}z) v\qquad \forall z,v\in\R^{2d}\ \text{with}\ v\cdot Mv=1\,.
\end{equation}
 Let us recall \cite[Proposition 4]{MonmarcheContraction}:
 
\begin{lemma}[from Proposition 4 of \cite{MonmarcheContraction}]\label{lemProp4}
For $\Lambda \geqslant \ell>0$ and $\gamma > \sqrt{\Lambda} - \sqrt{\ell}$, there exist $a,\rho>0$ and $c\in\R$ with $a > c^2$ such that for all $x=(q,p)\in\R^{2d} $ such that $\sigma(\na^2V(q))$ the spectrum of $\na^2 V(q)$ lies in $[\ell,\Lambda]$, considering~\eqref{locnab} and~\eqref{locM}, for any $v\in\R^d$ with $v\cdot M v= 1$,
 \[ v \cdot M \na b(x) v \leqslant - \rho\,.\]
 \end{lemma}
 
 \begin{remark}\label{remloccond}
The condition that $\gamma > \sqrt{\Lambda} - \sqrt{\ell}$ is sharp, in the sense that if it fails, there are potentials $V$ with $\sigma(\na^2 V(q)) \subset [\ell,\Lambda]$ such that the dissipativity condition is false for any choice of $M$, see \cite{MonmarcheContraction}. 

Alternatively, there is no restriction on $\gamma$ when $\na^2 V$ is constant, since it is then sufficient to work independently on each  of its eigenspaces. More generally, when (possibly up to an orthonormal change of variable) $V(q) = \sum_{i=1}^q V_i(q_i)$ for some 1D potentials $V_i$, then the condition that $\gamma > \sqrt{\Lambda_i} - \sqrt{\lambda_i} $ for all $i\in\cco 1,d\ccf$, where $V_i''(q) \in [\ell_i,\Lambda_i]$ for all $q\in\R$, is sufficient to get the same result as in Lemma~\ref{lemProp4}.
 \end{remark}

Thanks to Lemma~\ref{lemProp4}, assuming that 
\begin{equation}
\label{loctildeF}
F(q,p) = -\na V(q) - \gamma p + \tilde F(q,p)
\end{equation}
where $\sigma(\na^2 V(q)) \subset [\ell,\Lambda]$,  $\gamma > \sqrt{\ell} - \sqrt{\lambda}$ and $\na \tilde F \rightarrow 0$ at infinity, we can take $M$ as in Lemma~\ref{lemProp4} (see \cite{MonmarcheContraction} for explicit expressions for $M$ and $\rho$; in the particular case where $\gamma^2 \geqslant 4\Lambda$ for instance, it is possible to take $a=\Lambda^{-1}$, $c=\gamma^{-1}$ and then $\rho = \ell/(3\gamma) $) and then take $\eta$ such that
 \[\eta(z) =  \sup \{ v \cdot M \na b(M^{-1/2}z) v,\  z,v\in\R^{2d}\ \text{with}\ v\cdot Mv=1\}\,.\]
This choice ensures that $\eta(z) \leqslant - \rho/2$ outside a compact set, which allows to apply Theorems~\ref{thm:FerreStoltz} and~\ref{thm:final}.  

\begin{remark}
The condition that $\na \tilde F$ vanishes at infinity excludes the case where $F(q,p) = -\na V(q) - \gamma p$ with $V$ only convex outside a compact set. Indeed, decomposing $V=V_0+V_1$ with $V_0$ strongly convex would leave  $\tilde F(q,p) = -\na V_1(q)$, whose derivative doesn't vanish as $p\rightarrow \infty$. The same issue appeared in \cite[Theorem 2]{monmarche2025logarithmic}.
\end{remark}

In view of this remark, although it is in principle possible to fine-tune a force $\tilde F$ such that Theorem~\ref{thm:final} applies in the case~\eqref{loctildeF} while there is no global dissipativity (i.e. $\sup \eta \geqslant 0$), which theoretically gives a first example where the $L^p$ contraction~\eqref{eq:contraction_def} holds in a non-elliptic non-globally dissipative case, this leads to very specific and unnatural conditions (like $\tilde F(q) = -\chi(p) \na V_1(q) $ with a vanishing $\chi$). For instance, if we simply assume that $|\na \tilde F|\leqslant \delta$, then for $\delta$ small enough Lemma~\ref{lemProp4} gives global dissipativity, in which case our approach is not necessary. For this reason, we will not detail the application of Theorem~\ref{thm:final}  to this situation. This unsatisfactory conclusion highlights that Theorem~\ref{thm:final} is based on too rough bounds and that, as discussed at the end of Section~\ref{subsec:resultFerreStoltz}, it might be necessary to better take into account the behaviour at infinity of the process, either by controlling the leading eigenvector of the corresponding Feynman-Kac operator (at least in some directions, to get~\eqref{locGCR} with a supremum over some unbounded sets, e.g. $\mathcal B(0,R_0)\times \R^d$) or conduct a more model-specific analysis of the evolution of the distance, splitting the space in more areas than simply $\mathcal B(0,R)$ and $\R^d\setminus \mathcal B(0,R)$ for some $R>0$, as in \cite{chak2025reflection} (see especially Theorem~4.7 there). It may also be important in this kinetic case to avoid resorting to an upper bound $\eta$ satisfying~\eqref{eq:eta}, as discussed in Section~\ref{subsec:sharperDecompose}. Following these directions is beyond the scope of the present work.

\begin{remark}\label{rem:appliPropkin}
On this topic, let us apply Proposition~\ref{prop:sharppinfini} in the kinetic case~\eqref{eq:kinLangevin}. It is clear that constant paths of the form $\gamma_q(t) = (q,0)$ for a fixed $q\in\R^d$ are admissible. Consider for instance the equilibrium case where $F(q,p)= -\na V(q) - \gamma p$, with $\|\na^2 V\|_\infty <\infty$. Then the assumptions of Proposition~\ref{prop:sharppinfini} hold and we deduce that
\[ \lim_{p\rightarrow\infty} \kappa_p(t) \geqslant \overline{\kappa}_\infty(t) \geqslant \sup_{q\in\R^d} |\Phi_t^{\gamma_q}|\,, \]
where $\Phi_t^{\gamma_q} = \exp(t \na b(q,0))$. The eigenvalues (possibly complex) of $\na b(q,0)$ being $-\gamma/2 \pm \sqrt{\gamma^2/4 - \xi }$ with $\xi \in \sigma(\na^2 V(q))$, denoting  $\xi_0 = \inf\{\xi,\ \xi\in \sigma(\na^2 V(q)),\ q\in\R^d\}$, we get that
\[\frac1t \ln \overline{\kappa}_\infty(t) \geqslant \left\{\begin{array}{ll}
\sqrt{\gamma^2/4-\xi_0} - \gamma/2 & \text{if }\xi_0 <\gamma^2/4\\
-\gamma/2 & \text{otherwise.}
\end{array}\right.
 \]
 This shows that,   for all $t\geqslant 0$, when $\xi_0 <0$ (i.e. when $V$ is not convex) then $\kappa_p(t) \geqslant 1$ for $p$ large enough. However, when $\xi_0>0$, this result doesn't prevent $\kappa_p(t) < 1$ for arbitrarily large $p$, whatever the value of $\gamma$, even when the condition that $\gamma \geqslant \sqrt{\Lambda} - \sqrt{\ell}$ as discussed in Remark~\ref{remloccond} fails, so that there doesn't exist any symmetric positive matrix $M$ such that the norm $\|\cdot \|_M$ is deterministically contracted along a synchronous coupling. This is another hint that working with a local upper bound $\eta$ satisfying~\eqref{eq:eta}, even up to a change of coordinates, might be too rough in the kinetic case.
\end{remark}

\subsection{Colored noise diffusion}

Consider the process $(q_t,w_t) \in \R^{n+m}$ solving the SDE
\begin{equation}
\label{eq:colored}
\left\{\begin{array}{rcl}
\dd q_t &=& -\na V(q_t) \dd t + Aw_t \dd t \\
\dd w_t &= &-  w_t \dd t + \sigma_0 \dd B_t\,,
\end{array}\right.
\end{equation}
where $A$ is a surjective $n\times m$ matrix and $\sigma_0 $ is a non-singular $m\times m$ matrix. This model appears to model turbulences, finance or coarse-grained molecular dynamics, see e.g. \cite{10.1214/13-PS218,hanke2024stochastic} and references within. The variable $q_t$ describes the position of a particle in a potential $V$ subject to a colored noise $w_t$ whose spectrum can be prescribed with a suitable choice of $\sigma_0$.

Setting $Y_t=q_t$ and $Z_t = w_t+ \eta A^T q_t$ for some $\eta>0$,
\begin{align*}
\dd Y_t &= \co -\na V(Y_t)   + A(Z_t - \eta A^T Y_t) \cf \dd t  \\
\dd Z_t &= \co -\eta A^T  \na V(Y_t)   + \po \eta A^T A - \mathrm{Id}_d\pf (Z_t - \eta A^T Y_t) \cf \dd t + \sigma_0 \dd B_t\,.
\end{align*}
Assume that there exist $\rho,L,R>0$ such that, for all $y,y'\in\R^n$,
\begin{equation}
\label{eq:locnaV}
- \po \na V(y) - \na V(y') \pf \cdot (y-y') \leqslant \left\{\begin{array}{rl}
-\rho  |y-y'|^2  & \text{if } |y-y'|\geqslant R\\
L|y-y'|^2 & \text{otherwise,}
\end{array}\right.
\end{equation}
which is in particular the case if $V$ is convex outside a compact set. It is then easily seen that, for two solutions $(q_t,w_t)$ and $(q_t',w_t')$ driven by the same Brownian motion,
\[\dd \po |q_t-q_t'|^2 + C_0 |w_t-w_t'|^2 \pf \leqslant - \rho'  \po |q_t-q_t'|^2 + C_0 |w_t-w_t'|^2 \pf \dd t\]
as long as $|q_t-q_t'|^2 + C_0 |w_t-w_t'|^2 \geqslant R'$, for some $\rho',R'>0$, by taking $C_0$ large enough. This is exactly saying  that Assumption~\ref{assum:Contract_at_infty} holds for the process $(Y_t,Z_t)$ (obtained from $(q_t,w_t)$ by a linear change of variables).

Writing $b_1(y,z) = -\na V(y) + A(z-\eta A^T y)$, for $x=(y,z),x'=(y',z')\in\R^{n+m}$, 
\[\po b_1(x)-b_1(x')\pf \cdot (y-y') \leqslant L |y-y'|^2 + |A||z-z'||y- y'| - \eta |A^T (y-y')|^2\,.  \]
Since $A^T$ is injective, we can take $\eta$ large enough so that $\eta |A^T y | \geqslant 2L|y|^2$ for all $y\in\R^n$. As a consequence, Assumption~\ref{assum:decomposeYZ} holds (the other conditions being readily checked).

We have thus shown that Theorem~\ref{thm:contractionMarkov} applies and gives the following:

\begin{proposition}
Assuming~\eqref{eq:locnaV} for some $\rho,L,R>0$ then, for any $p\geqslant 1$, the process~\eqref{eq:colored} satisfies the contraction~\eqref{eq:weightWp} with the semi-metric $\omega(x,x')=\max(|x-x'|,|x-x'|^p)$ for some $C_p,\lambda_p>0$.
\end{proposition}

Let us now discuss the issue when attempting to apply Theorems~\ref{thm:FerreStoltz} and \ref{thm:final}  to the process~\eqref{eq:colored}. Denoting by $b$ the drift of~\eqref{eq:colored},
\[\na b(x) = \begin{pmatrix}
-\na^2 V(q)& A \\
0 & - \mathrm{Id}
\end{pmatrix}\,. \]
When $V$ is strongly convex, it is clear that a global dissipativity condition is satisfied. When $V$ is strongly convex only outside a compact set however, we are faced to similar difficulties as discussed in the previous section for the kinetic Langevin process. In particular, it is not possible to find a symmetric positive matrix $M$ and a function $\eta$ satisfying both~\eqref{locetaM} and that $\eta$ is negative outside a compact set, the issue being that $\na^2 V(q)$ may not be positive with $w$ arbitrarily large. Indeed, for such a $q$, taking $v=(v_1,0) \in \R^{n+m}$ with $v_1$ an eigenvector of $\na^2 V(q)$ associated to an eigenvalue $-\nu<0$, for any $M$ and $w\in\R^m$,
\[v\cdot M \na b(q,w) v = \nu v\cdot M v\,.\]

As in the previous section, addressing this issue is beyond the scope of the present work. The process~\eqref{eq:colored} having a simpler structure than~\eqref{eq:kinLangevin} (with the auxiliary variables $w$ evolving independently from $q$), the required analyis might be simpler. The  model~\eqref{eq:colored}  could thus serve as a prototypical problem to get an $L^p$ Wasserstein contraction in a non-elliptic, non-globally dissipative case, under reasonable conditions (which in the present case would be simply that $V$ is strongly convex outside a compact set, with a small negative global lower bound $-\delta$ on $\na^2 V$ as in the elliptic case described in Proposition~\ref{prop:elliptic}).

\subsection{Morris-Lecar neuron model}

The stochastic Morris-Lecar model studied in \cite{HerdaPerthameMonmarche} is the process $(v_t,g_t)_{t\geqslant0}$ on $[V_L,V_E]\times\R_+$ solving the SDE
\begin{equation}\label{eq:SDEML}
\left\{\begin{array}{rcl}
\dd v_t & = & g_L(V_L-v_t) \dd t + g_t(V_E-v_t)\dd t \, ,\\
\dd g_t  & = & \gamma\po G(v_t)-g_t\pf\dd t + \sqrt{2}a  \dd B_t + \dd L_t\,,
\end{array}\right.
\end{equation}
with parameters $a\geqslant 0$, $0<V_L < V_E$, $\gamma,g_L,G_M>0$, $G\in\mathcal C^2 ([V_L , V_E ]; (0,G_M])$, where $B$ is a one-dimensional Brownian motion and $(L_t)_{t\geqslant 0}$ is the local time of $(g_t)_{t\geqslant 0}$, which amounts to say that $g_t$ is reflected at $0$. This  describes the evolution of the voltage $v_t$ and conductance $g_t$ of an electric cell.

Under suitable conditions, a Wasserstein 1 contraction is proven in \cite{HerdaPerthameMonmarche} for this model, with a coupling argument similar to the proof of Theorem~\ref{thm:contractionMarkov}. This illustrates  how all the results stated in the present work can easily be generalized to processes which are reflected at the boundary of a convex domain. Indeed, as noticed already in~\cite{wang1994application}, in this situation, the contribution of the reflections to the evolution of the distance between the coupled processes is always negative.

\section{On the general results}\label{sec:proofGeneral}

\subsection{Sharper analysis of the Lyapunov exponent}\label{subsec:sharperDecompose}

Let us briefly discuss how it would be possible in dimension $d>1$ to study $\kappa_p(t)$ without using a bound $\eta$ satisfying~\eqref{eq:eta}. Decompose the process $V$ solving~\eqref{eq:EDO_Phi} as $V_t = R_t \theta_t$ with $R_t=|V_t|$. Then
\begin{equation}\label{eq:theta2}
\dd R_t = \theta_t \cdot \na b(X_t)\theta_t  R_t\dd t \qquad \dd \theta_t = \co \na b(X_t) -  \theta_t \cdot \na b(X_t)\theta_t \mathrm{Id}\cf \theta_t \,. 
\end{equation}
The first equation is solved explicitly in terms of $R_t$ and gives the representation
\begin{equation}
\label{eq:locaLyap}
 \mathbb E_x \po |\Phi_t^X v|^p \pf  =     \mathbb E_{x,v}\co \exp\po p\int_0^t \theta_s \cdot \na b(X_s)\theta_s \dd s \pf\cf\,,
\end{equation}
with the subscript $v$ indicating the initial condition of the solution $\theta$ of~\eqref{eq:theta2}.  

In general, the Markov process $(X_t,\theta_t)_{t\geqslant 0}$ might not be ergodic even if $X$ is. For instance, for the Ornstein-Uhlenbeck process, $b(x)=-x$ and thus $\theta_t = \theta_0$ for all $t\geqslant 0$. In general, for a fine study of $\kappa_p(t)$, we might study this Markov process on each of its recurrence classes. We can decompose $\R^d$ as a supplementary sum of linear subspaces $E_1,\dots,E_k$ which are each stable by $\{\na b(x),\ x\in\R^d\}$ and study separately the cases where $\theta_0= v$ is initially in each of these subspaces (so that $(X,\theta)$ is irreducible on $\R^d\times E_i$ if $X$ hypoelliptic and controllable). To deal then with the supremum in $v$ in the definition of $\kappa_p(t)$, considering a basis $(e_1,\dots,e_d)$ and  denoting by $v_1,\dots,v_d$ the coordinates of a vector $v\in\Sd$ in this basis,
\[ \sup_{v\in\Sd} \po \mathbb E_{x}\po |\Phi_t^X(v)|^p\pf\pf^{1/p}  \leqslant  \sum_{i=1}^d |v_i| \po \mathbb E_{x}\po |\Phi_t^X(e_i)|^p\pf\pf^{1/p} \leqslant C_e  \max_{i\in\cco 1,d\ccf} \po \mathbb E_{x}\po |\Phi_t^X(e_i)|^p\pf\pf^{1/p}\,, \]
with $C_e = \sup\{\sum_{i=1}^d|v_i|,\ v\in\Sd\} $. It is thus sufficient to show that $\sup_{x\in\R^d}\mathbb E_{x}\po |\Phi_t^X(e_i)|^p\pf$ goes to zero for each $i\in\cco 1,d\ccf$ to get the same result for $\kappa_p(t)$.

Nevertheless, as mentioned in Section~\ref{sec:generalresult}, we won't continue this analysis, focusing on the study of $G_p^*$ with a bound $\eta$ such that~\eqref{eq:eta} holds.

\subsection{Proof of Propositions~\ref{prop:COntractionkappap} and~\ref{prop:dim1} } \label{sec:proofProp1}

We start by linearizing the equation:

\begin{lemma}\label{lem:synchronous}
Under Assumption~\ref{Assum:basic}, for any $t\geqslant 0$ and $x,y\in\R^d$,
\[\left| \co \mathbb E_{(x,y)}\po |X_t-Y_t|^p \pf \cf^{1/p} -  \co \mathbb E_{x}\po |\Phi_t^X (x-y)|^p \pf \cf^{1/p} \right| \leqslant  c_t  |x-y|^2\,,\]
with  $c_t = \frac12 \|\na^2 b\|_\infty e^{2Lt}\po e^{2Lt}-1\pf$ and $(X,Y)$ as in Proposition~\ref{prop:COntractionkappap}.
\end{lemma}
\begin{proof}
Let $x,y\in\R^d$ with $x\neq y$. Consider  a synchronous coupling
\begin{align*}
\dd X_t & = b(X_t) \dd t + \sigma \dd B_t\\
\dd Y_t & = b(Y_t) \dd t + \sigma \dd B_t\\
\dd V_t & = \na b(X_t) V_t \dd t \,,
\end{align*}
with $X_0=x$, $Y_0=y$ and $V_0 = (y-x)/|y-x|$.  Set $\delta = |x-y|$. From Assumption~\ref{Assum:basic}
\[\dd |X_t - Y_t|^2 \leqslant 2 L |X_t-Y_t|^2\dd t \,,\]
so that  $|X_t-Y_t|^2 \leqslant \delta^2 e^{2Lt}$ for all $t\geqslant 0$ almost surely. Besides,
\begin{align*}
\dd (X_t + \delta V_t - Y_t) &= \co b(X_t) - b(Y_t) + \delta \na b(X_t) V_t \cf \dd t \\
&= \na b(X_t) (X_t + \delta V_t - Y_t)\dd t + R_t\dd t 
\end{align*}
where
\[R_t = b(X_t)-b(Y_t) - \na b(X_t)(X_t-Y_t) \] 
is such that
\[|R_t| \leqslant \frac12 \|\na^2 b\|_\infty |X_t-Y_t|^2 \leqslant \frac12 \|\na^2 b\|_\infty e^{2Lt}\delta^2\,. \]
By Grönwall Lemma, using that $X_0 + \delta V_0 - Y_0 = 0$,
\[|X_t + \delta V_t - Y_t| \leqslant \frac12 \|\na^2 b\|_\infty e^{2Lt}\po e^{2Lt}-1\pf \delta^2 = c_t\delta^2 \,. \] 
By triangular inequality,
\[\delta |\Phi_t^Xv| -  |X_t + \delta V_t - Y_t|  \leqslant |X_t-Y_t| \leqslant \delta |\Phi_t^Xv| + |X_t + \delta V_t - Y_t|\,,\]
from which
\[\delta \po \mathbb E\po |\Phi_t^Xv|^p\pf\pf^{1/p} - c_t \delta^2 \leqslant  \co \mathbb E\po |X_t-Y_t|^p\pf\cf^{1/p} \leqslant \delta \po \mathbb E\po |\Phi_t^Xv|^p\pf\pf^{1/p} + c_t \delta^2\,, \]
which concludes the proof of the lemma.
\end{proof}

\begin{proof}[Proof of Proposition~\ref{prop:COntractionkappap}]
 For fixed $x,y\in\R^d$ and $N\geqslant 1$, set $z_k=(1-k/N)x+k/Ny$ for $k\in\cco 0,N\ccf$. For a fixed Brownian motion $B$, denote by $X^{(k)}$ the solution of~\eqref{eq:EDS_X}  starting at $X_0^{(k)}=z_k$. Thanks to Lemma~\ref{lem:synchronous},
 \begin{align*}
 \co \mathbb E_{(x,y)}\po |X_t-Y_t|^p\pf\cf^{1/p} & \leqslant \sum_{k=1}^N  \co \mathbb E\po |X_t^{(k-1)}-X_t^{(k)}|^p\pf\cf^{1/p}  \\
 & \leqslant \sum_{k=1}^N \co \kappa_p(t)  |z_k-z_{k-1}| + c_t  |z_k-z_{k-1}|^2 \cf \\
 & = \kappa_p(t) |x-y| + \frac{c_t}{N}|x-y|^2\,.
 \end{align*}
 Letting $N\rightarrow \infty$ shows that
 \[\kappa_p^p(t) \leqslant \sup_{x,y\in\R^d,x\neq y} \frac{\mathbb E_{(x,y)}\po |X_t-Y_t|^p \pf}{|x-y|}\,,\]
Conversely, Lemma~\ref{lem:synchronous} shows that, for any $x\in\R^d,v\in\Sd$ and $\varepsilon>0$,
\begin{align*}
\co \mathbb E_{x}\po |\Phi_t^X v |^p \pf \cf^{1/p} & \geqslant \varepsilon^{-1} \co \mathbb E_{(x,x+ \varepsilon v)}\po |X_t-Y_t|^p \pf \cf^{1/p} - c_t \varepsilon\,.
\end{align*}
Taking the supremum over $x\in\R^d,v\in\Sd$ and then taking $\varepsilon \rightarrow 0$ concludes.
\end{proof}

\begin{proof}[Proof of Proposition~\ref{prop:dim1}]
Let $x,y\in\R$. For  $(X_t,V_t,Y_t)_{t\geqslant 0}$ as in the proof of Lemma~\ref{lem:synchronous},
\begin{equation}
\label{loc1}
\mathcal W_1(\delta_x P_t,\delta_y P_t) \geqslant  |\mathbb E_x \po X_t \pf - \mathbb E_y(Y_t)| \geqslant |x-y||\mathbb E\po V_t\pf| - c_t|x-y|^2\,.
\end{equation}
In dimension 1, $\Phi_t^X =  e^{\int_0^t b'(X_s)\dd s }$ and, since $V_0\in\{-1,1\}$, $|\mathbb E\po V_t\pf| = \mathbb E(\Phi_t^X )$. Dividing~\eqref{loc1} by $|x-y|$ and letting $y\mapsto x$ shows that, for any $x\in\R^d$,
\[\sup_{y\in\R^d\setminus\{x\}} \frac{\mathcal W_1(\delta_x P_t,\delta_y P_t)}{|x-y|} \geqslant \mathbb E_x \po e^{\int_0^t b'(X_s)\dd s }\pf\,.\]
Taking the supremum over $x\in\R$ and conditionning with respect to the initial condition shows that 
\[\kappa_1(t) \leqslant \sup_{x,y\in\R,x\neq y} \frac{\mathcal W_1(\delta_x P_t,\delta_y P_t)}{|x-y|} = \beta_1(t)\,.\]
Proposition~\ref{prop:COntractionkappap} provides the converse inequality, which concludes.
\end{proof}

\subsection{Proofs of the negative results}\label{subsec:negativeproof}

\begin{proof}[Proof of Proposition~\ref{prop:tempscourt}]
This essentially follows from   \cite[Proof of $(v)\Rightarrow (i)$]{VonRenesseSturm}, or its extension to non-reversible cases in \cite[Proof of $(vi)\Rightarrow(i)$]{MonmarcheBruit}. Let us repeat the argument for the reader's convenience and adapt them to the specific way Proposition~\ref{prop:tempscourt} is stated.

Since $\beta_1(t) \leqslant \beta_p(t) \leqslant e^{-t\lambda^*}$ for all $t\geqslant 0$, it only remains to prove that
\[\beta_1(t) \geqslant 1 -  t \lambda_* + \underset{t\rightarrow 0}o(t)\,,\]
or in other words that $\liminf (\beta_1(t) - 1 + t \lambda^*)/t \geqslant 0$ as $t\rightarrow 0$.

Fix $\varepsilon\in(0,1]$. Let  $x\in\R^d,v\in\Sd$ be such that $v\cdot \na b(x)v  \geqslant -\lambda^* - \varepsilon$. Let $\mathcal U$ be a neighborhood of $x$ such that $v\cdot \na b(y)v \geqslant -\lambda^*-2 \varepsilon$ for all $y\in\mathcal U$. Let $\theta\in\mathcal C^\infty(\R)$ be such that $\theta(r)=r$ for $r\in[-1,1]$, $\|\theta'\|_\infty \leqslant 1$, $\theta(r)=0$ for $r\notin [-3,3]$ (in particular, $\|\theta^{(k)}\|_\infty < +\infty$ for all $k\geqslant 0$). Consider the function $f$ on $\R^d$ given by $f(y) = \theta((y-x)\cdot v)$ for all $y\in\R^d$, which is $\mathcal{C^\infty}$ and $1$-Lipschitz. For all  non-negative and compactly supported $\varphi \in \mathcal C^\infty(\R^d)$,
\[K_{\varphi}(t) \ :=\ \int_{\R^d} |\na P_t f|^2(y) \varphi(y)\dd y \ \leqslant \ \beta_1^2(t) \int_{\R^d}   \varphi(y)\dd y \,.\]
Besides, for all $y\in\mathcal B(x,1/2)$, $\na f(y) = v$ and $\na^2 f(y)=0$. Fixing $\varphi\in\mathcal C^\infty(\R^d)$ non-zero, non-negative with compact support included in $\mathcal U\cap \mathcal B(x,1/2)$, we get that  $K_\varphi(0) = \int_{\R^d} \varphi(y)\dd y$ and, writing $\mathcal L=b\cdot \na +\Delta$ the generator of~\eqref{eq:EDS_X},
\begin{align*}
K_\varphi'(0) &= 2 \int_{\R^d} \na   f(y) \cdot \na \mathcal L   f(y) \varphi(y)\dd y \\
& = 2 \int_{\R^d}    \na   f(y) \cdot  \po \na b(y) \na f(y) + \mathcal L  (\na f)(y)\pf \varphi(y)\dd y\\
& \geqslant -2(\lambda^* + 2\varepsilon) \int_{\R^d} \varphi(y)\dd y \,. 
\end{align*}
As a consequence, 
\[ \beta_1^2(t) \geqslant \frac{K_\varphi(t)}{K_\varphi(0)} \geqslant 1 + t \frac{K_\varphi'(0)}{K_\varphi(0)} + \underset{t\rightarrow 0}o(t) \geqslant  1 - 2t(\lambda^*+2\varepsilon) + \underset{t\rightarrow 0}o(t)  \,.\]
This concludes, since  $\varepsilon$ is arbitrary.
\end{proof}

\begin{proof}[Proof of Proposition~\ref{prop:sharppinfini}]
Let $T>0$, $\gamma\in\mathcal A(T)$ and $x\in\R^d$. Let $R>0$ be large enough so that $|\gamma(t)|+1\leqslant R$ for all $t\in[0,T]$ and, moreover,  if $\na b$ is not lower bounded,  so that for all $t,\varepsilon>0$,
\[\mathbb P_x\po|X_t-\gamma(0)|\leqslant \varepsilon\, \text{and}\, |X_s|\leqslant R\ \forall s\in[0,t]\pf >0\,.\]
 Let $\varepsilon\in(0,1]$. Let $X$ be a solution of~\eqref{eq:EDS_X} with $X_0=x$, and  $Y$ be the solution of~\eqref{eq:EDS_X} (with the same $B$ as $X$) initialised at time $\varepsilon$ to $Y_\varepsilon = \gamma(\varepsilon)$. Consider the events
\[\mathcal B_1 = \{|X_\varepsilon - \gamma(\varepsilon)|\leqslant \varepsilon\,\text{and}\, |X_s|\leqslant R\ \forall s\in[0,\varepsilon]\}\,,\qquad \mathcal B_2  = \{ |Y_t-\gamma(t)| \leqslant \varepsilon\ \forall t\in[\varepsilon,T]\}\,.\]
 By assumption and by definition of $\mathcal A(T)$, both $\mathcal B_1$ and $\mathcal B_2$ have positive probability. Moreover, by the independence of the Brownian increments, these two events are independent, from which $\mathbb P(\mathcal B )>0$ with $\mathcal B=\mathcal B_1 \cap \mathcal B_2$. Since $|X_t-Y_t| \leqslant e^{LT}|X_\varepsilon-Y_\varepsilon|$ for all $t\in [\varepsilon,T]$ almost surely, under $\mathcal B$,
\[ |X_t-\gamma(t)| \leqslant \varepsilon' := (1 + e^{LT})(\varepsilon+ |\gamma(0)-\gamma(\varepsilon)|)\qquad \forall t\in[\varepsilon,T]\,.\]
In particular, for $\varepsilon$ small enough (which we assume further on), this implies that $|X_t|\leqslant |\gamma(t)|+1\leqslant R$ for all $t\geqslant \varepsilon$, hence for all $t\in[0,T]$ under $\mathcal B$ (if $\na b$ is not bounded). Let $M=\sup_{|x|\leqslant R}|\na b(x)|$ if $\na b$ is not bounded and $M=\|\na b\|_\infty$ otherwise. In both situations, under $\mathcal B$, $|\na b(X_t)|\leqslant M$ and $|\na b(\gamma(t))|\leqslant M$ for all $t\in[0,T]$.

Let $V_t$ and $w_t$ solve $\dd V_t = \na b(X_t)V_t\dd t$ and $\partial_t w_t = \na b(\gamma(t)) w_t$ for $t\in[0,T]$ with $V_0=w_0=v$ where $v\in\Sd$ is such that $|\Phi_T^\gamma|=|\Phi_T^\gamma v|$ (and thus, $|\Phi_T^\gamma|=|w_T|$). Under $\mathcal B$, 
\[|V_t-w_t|\leqslant \int_0^\varepsilon |\na b(X_s)-\na b(\gamma(s))|\dd s \leqslant 2M \varepsilon\,, \] 
and  for $t\geqslant \varepsilon$,
\[|V_t - w_t| \leqslant |V_\varepsilon - w_\varepsilon| +  \|\na ^2 b\|_\infty \varepsilon' \int_\varepsilon^t |w_s|\dd s + M \int_{\varepsilon}^t |V_s-w_s|\dd s \,, \]
from which
\[|V_T- w_T| \leqslant \varepsilon'' := e^{M T} \co 2M\varepsilon+  \|\na ^2 b\|_\infty \varepsilon' \int_0^T |w_s|\dd s\cf \,.\]
As a consequence, under $\mathcal B$,
\[
|V_T| \geqslant |w_T| - |V_T-w_T|  \geqslant  |\Phi_T^\gamma| - \varepsilon''  \,,\]
and thus, for any $p\geqslant 1$,
\[\po \mathbb E_x\po |V_t|^p\pf \pf^{1/p} \geqslant \po \mathbb P(\mathcal B)\pf^{1/p} \po |\Phi_T^\gamma|- \varepsilon''\pf\,.\]
Letting $p\rightarrow \infty$ and then $\varepsilon\rightarrow 0$ concludes the proof, since $\gamma\in\mathcal A(T)$ is arbitrary.
\end{proof}

\subsection{Proof of Theorem~\ref{thm:FerreStoltz}}\label{sec:preuveThemFerreStoltz}

We separate the proof in three parts. First we explain how the first item of Theorem~\ref{thm:FerreStoltz} follows from~\cite{FerreStoltz}. Second, we prove the upper bound~\eqref{eq:limsupGJpeta}, and third the lower bound~\eqref{e:liminfGJpeta}

\begin{proof}[Proof of Theorem~\ref{thm:FerreStoltz} Part 1]
This is exactly the content of \cite[Theorem 1]{FerreStoltz}, except that the uniformity over compact sets is not stated there. However, it follows from its proof, as we detail now (referring to \cite{FerreStoltz} for details).

 First, let us notice that, following the computations in the proof of \cite[Proposition 1]{FerreStoltz}, Assumption~\ref{Assum:PointwiseContractInfty} implies that the Witten-Lyapunov condition in \cite[Assumption 3]{FerreStoltz} holds in our case with a Lyapunov function $W(x) = e^{\varepsilon |x|^2}$ for some small $\varepsilon>0$.

Then, in Step 5 of the proof of \cite[Theorem 1]{FerreStoltz}, it is established that, for a fixed $t_0>0$, for any $n\in\N$,
\begin{equation}
\label{locQGP}
G_p(x,nt_0) = e^{ n t_0 \mathcal J(p\eta)}  h(x) Q_h^n h^{-1}(x)\,,  
\end{equation}
where $h$ is a positive continuous function and $Q_h$ is a Markov operator such that $Q_h^n$ converges exponentially fast in $n$ towards a probability measure $\mu_h$ in a suitable norm. Specifically, \cite[Equation (89)]{FerreStoltz} shows that
\[|Q_h^n h^{-1}(x) - \mu_h(h^{-1})| \leqslant C e^{-\lambda n} W(x)\qquad \forall x\in\R^d\,,\]
for some constants $C,\lambda>0$. In particular,  for any compact set $\mathcal K\subset \R^d$, $\inf_{x\in\mathcal K}|Q_h^n h^{-1}(x)|>0$ for $n$ large enough and then, from~\eqref{locQGP},
\[\sup_{x\in \mathcal K}\left| \frac1{nt_0} \ln G_p(x,nt_0) - \mathcal J(p\eta) \right| \underset{n\rightarrow\infty}\longrightarrow 0\,. \]
The result with all $t\geqslant 0$ follows from the fact that, for $t\in[(n-1)t_0,nt_0]$,
\[G_p(x,nt_0)e^{-t_0\|\eta\|_\infty} \leqslant  G_p(x,t) \leqslant G_p(x,(n-1)t_0)  e^{t_0\|\eta\|_\infty} \,.\]
\end{proof}

\begin{proof}[Proof of the upper bound~\eqref{eq:limsupGJpeta}]
Given $X$ a solution of~\eqref{eq:EDS_X}, denote $\tau = \inf\{t\geqslant 0,|X_t|\leqslant R\}$, where $R$ is as in Assumption~\ref{Assum:PointwiseContractInfty}. Then, by the strong Markov property, for any $x\in\R^d$,
\begin{align}
G_p(x,t)  &\leqslant \mathbb E_x \co e^{-\rho t} \1_{\tau>t} + e^{-\rho \tau}\1_{\tau \leqslant t}  \sup_{y\in\mathcal B(0,R)} G_p (y,t-\tau)   \cf \nonumber \\ 
&\leqslant e^{C_R} \mathbb E_x \co e^{-\rho t} \1_{\tau>t} + e^{-\rho \tau}\1_{\tau \leqslant t} e^{(t-\tau)\mathcal J(p\eta)}   \cf \label{locFGH}\\ 
& \leqslant  \exp \po C_R+ t \max  (-\rho,\mathcal J(p\eta)) \pf \,, \nonumber
\end{align}
thanks to the first part of Theorem~\ref{thm:FerreStoltz}. Taking the supremum over $x\in\R^d$ concludes the proof of~\eqref{eq:limsupGJpeta}.
\end{proof}

\begin{proof}[Proof of the lower bound~\eqref{e:liminfGJpeta}]
From the first part of Theorem~\ref{thm:FerreStoltz} we already have that 
\[  G_p^*(t)  \geqslant G_p(0,t) \geqslant     e^{    t  \mathcal J(p\eta) - C_0} \,.  \]
It remains to prove that $G_p^*(t) \geqslant e^{-\rho' t }/C$ for some $C>0$. For any $\varepsilon >0$, let $R_\varepsilon>0$ be such that $\eta(x) \geqslant -\rho'-\varepsilon$ for all $x\in\R^d$ with $|x|\geqslant R_\varepsilon$. Then for any $x\in\R^d$ we can bound
 \[G_p(x,t) \geqslant e^{-(\rho'+\varepsilon) t} \mathbb P_x\po |X_s| \geqslant R_\varepsilon\ \forall s\in[0,t] \pf\,.\]
 Writing $M_t = \sup_{s\in[0,t]}|B_s|$, for all $t\geqslant 0$,
 \[|X_t| \geqslant |X_0| - K t - K\int_{0}^t |X_s|\dd s - |\sigma| M_t\,.\]
 Applying the Grönwall lemma to $Y_s = X_{t-s}$ gives
 \[\forall t\geqslant 0\,,\qquad  |X_t| \geqslant e^{-Kt} |X_0| - K t - |\sigma| M_t\,,\]
 and thus, since $M_t$ is a.s. finite,
 \[G_p(x,t) \geqslant e^{-(\rho'+\varepsilon) t} \mathbb P\po   e^{-Kt} |x| - K t - |\sigma| M_t \geqslant R_\varepsilon \pf \underset{|x|\rightarrow \infty}\longrightarrow 1\,. \]
 Taking the supremum over $x\in\R^d$ thus concludes.
\end{proof}

\section{On the non-asymptotic results}\label{sec:proofCoupling}
\subsection{Modified cost contraction with a Markovian coupling}\label{sec:MarkovianCoupling}

This section is devoted to the proof of Theorem~\ref{thm:contractionMarkov}.  Assumptions~\ref{assum:decomposeYZ} and~\ref{assum:Contract_at_infty} hold, and we decompose the process $X$ solving~\eqref{eq:EDS_X} as $X_t=(Y_t,Z_t) \in \R^{n+m}$ according to Assumption~\ref{assum:decomposeYZ}. The proof, decomposed in six steps, relies on the definition of a  Markovian coupling of two processes $(X_t)_{t\geqslant 0}$ and $(X_t')_{\geqslant 0}$ associated to $(P_t)_{t\geqslant 0}$ and the construction of a suitable cost $\tilde \omega$, equivalent to $\omega$, and such that $\tilde \omega(X_t,X_t')$  decays exponentially on average along time.

Before entering  the technical details of the proof, let us provide a high-level informal presentation of the arguments.

\begin{itemize}
\item \emph{Calibrating the noise.} In view of the decomposition~\eqref{eq:decomposSigma}, we will only exploit the noise in $\dd (Z_t-Z_t')$,  which will be described by a function $\beta$ (see~\eqref{eq:couplage2}). When $\beta(X_t,X_t')=0$, the coupling is synchronous, and when $\beta(X_t,X_t')=\theta$, it is reflecting along the variable $z$ (which maximizes the quadratic variation of $|Z_t-Z_t'|$). When $Z_t=Z_t'$, stochastic fluctuations can only increase the distance between the processes, so that we switch them off by setting $\beta=0$ (since we will work with smooth coefficients, we will in fact set $\beta=0$ whenever $|Z_t-Z_t'|\leqslant\xi$ for some small $\xi>0$). Moreover, when the distance between the processes is very large, we want the cost $\omega(X_t,X_t')$ to be of order $|X_t-X_t'|^p$ with $p\geqslant 1$: by convexity, stochastic fluctuations  would increase the cost on average, so, again, we switch off the noise. By contrast, in intermediary regions, we will need the noise to bring together the $z$ variables, and thus we switch it on to $\beta=\theta$.

\item \emph{Balance between the two coordinates.} When $|Z_t-Z_t'|\leqslant \xi$, we cannot exploit on the noise to get a contraction, which means we can only rely on the deterministic contraction at rate $\rho_1$ in the $y$ variables from Assumptions~\ref{assum:decomposeYZ}. As a consequence, we need to work with a suitable distance $r_t= |Z_t-Z_t'|+M|Y_t-Y_t'|$ (up to regularization involving $\xi$) where $M$ is large enough so that, when the main contribution of the distance is from the $y$ variables, the negative term in~\eqref{locAssumYZ} dominates all the others.

\item \emph{Concave modification.} Following the method introduced in \cite{Eberle1} to get a Wasserstein contraction along reflection couplings, we consider a modification $f(r_t)$ of the distance with an increasing concave $f$, so that the stochastic fluctuations in $\dd r_t$ tend to decrease $f(r_t)$ on average.  We have designed $r_t$ so that, either $|Y_t-Y_t'|$ dominates $|Z_t-Z_t'|$ and $r_t$ has a negative drift, or in the second case there is noise which can be exploited with the concavity of $f$ (or, third case, $X_t=X_t'$ up to regularization and there is nothing to do).

\item \emph{Large distances.} At this stage, we have $f(r_t)$ which decays on average, but we need $f''$ to be sufficiently negative and this forces $f$ to be bounded. For this reason, for large distances, we have to count on the deterministic contraction from Assumption~\ref{assum:Contract_at_infty}. Writing $s_t = \|X_t-X_t'\|_Q$, this condition means that $s_t^p$ is deterministically contracted by the drift as long as $s_t \geqslant S^*$. Thus, we have to combine $f(r_t)$ for small distances with some $g(s_t)$ of order $s_t^p$ for large distances. The difficulty is that the stochastic fluctuations from the reflection coupling in the $z$ variables are increasing $s_t^p$ in average. As a consequence, we switch off this contribution for small distances $s \leqslant S^*$, settings $g(s)=0$. However, when we start switching it on, initially, $g''$ has to be positive, which gives on average a positive contribution to the evolution of $f(r_t) + g(s_t)$. Hence, for small values of $s-S^*$, we start by setting $g(s_t) = \varepsilon (s-S^*)_+^2$ with a sufficiently small $\varepsilon$ so that the contribution of $g''$ is absorbed by the decay from $f''(r_t)$ (which requires to keep the reflection coupling for these distances). With this, at some point, when $s - S^*$ is large enough, $g(s_t)$ have reached a sufficiently large value so that its exponential decay (induced by the drift thanks to Assumption~\ref{assum:Contract_at_infty}) is enough to counteract the part from $g''(s_t)$ and to give a global decay of $f(r_t)+g(s_t)$. This means we can switch of the noise, and then in the absence of stochastic fluctuations  there is no more any problem  to take $g''$ has large as desired, and thus we can take $g(s)$ to grow as $s^p$. 

Notice that the two distances $r_t$ and $s_t$ are equivalent but not equal and thus we have to switch off the noise at some level $R_1^*$ such that $r_t \geqslant R_1^*$ implies that $s_t \geqslant S_1^*$ where $S_1^*$ is chosen so that $g'(S_1^*)$ is large enough (to give a dominant negative term in $\dd (f(r_t)+g(s_t))$, in particular to control $g''(s_t)$), and similarly we can take $g''(s)$ very large only for $s\geqslant S_2^*$ with $S_2^*$ large enough so that $s_t \geqslant S_2^*$ implies that $r_t \geqslant R_1^*$ and thus the noise in $\dd(X_t-X_t')$ is already switched off at that point.
\end{itemize}

We now turn to the formal proof of Theorem~\ref{thm:contractionMarkov}.

\medskip

\noindent\textbf{Step 1. Definition of the coupling}
We consider a coupling $(X_t,X_t')=(Y_t,Z_t,Y_t',Z_t')$ solving
\begin{equation*}
\left\{
\begin{array}{rcl}
\dd Y_t &=& b_1(X_t) \dd t + \tilde \sigma_1 \dd B_t'' \\
\dd Z_t & =& b_2(X_t) \dd t + \beta(X_t,X_t')  \dd B_t + \sqrt{\theta^2-\beta^2(X_t,X_t')} \dd B_t' + \tilde \sigma_2 \dd B_t''  \\
\dd Y_t' &= &b_1(X_t') \dd t + \tilde \sigma_1 \dd B_t'' \\
\dd Z_t' & = & b_2(X_t') \dd t + \beta(X_t,X_t') \po \mathrm{Id}_m - 2e_t e_t^T \pf \dd B_t + \sqrt{\theta^2-\beta^2(X_t,X_t')} \dd B_t' + \tilde \sigma_2 \dd B_t'' \,, 
\end{array}
\right.
\end{equation*}
where $e_t = (Z_t-Z_t')/|Z_t-Z_t'|$ when $Z_t\neq Z_t'$ and is $(1,0,\dots,0)$ otherwise, $B$ and $B'$ (resp. $B''$) are two independent Brownian motion in $\R^m$ (resp. $\R^{n+m}$) and $\beta\in\mathcal C^2(\R^{n+m}\times\R^{n+m},[0,\theta])$. The function $\beta$ will be specified below; for now, let us notice that it will be chosen so that $\beta((y,z),(y',z'))=0$ when $|z-z'|\leqslant \xi/2$ for some regularization parameter $\xi\in(0,1]$   (that will be sent to zero at the end). As a consequence, the drift and diffusion coefficients of the proces $(X_t,X_t')$ are $\mathcal C^2$. Together with the one-sided Lipschitz condition in Assumption~\ref{assum:decomposeYZ}, this ensures that the process is well-defined. Moreover, by the Lévy characterization of the Brownian motion, using in particular that
\[\beta^2(x,x') \po \mathrm{Id}_m - 2 e e^T \pf^2 + \po \theta^2 - \beta^2 (x,x') \pf \mathrm{Id}_m \]
is orthogonal for all $x,x'\in\R^{n+m}$ and $e \in \mathbb S^{m-1}$,  we get that $(X_t)_{t\geqslant 0}$ and $(X_t')_{t\geqslant 0}$ are both processes solving the SDE~\eqref{eq:EDS_X} (driven by some Brownian motion).

Along the coupling,
\begin{equation}
\label{eq:couplage2}
\left\{
\begin{array}{rcl}
\dd (Y_t-Y_t') &=& \po b_1(X_t) - b_1(X_t')\pf  \dd t \\
\dd (Z_t-Z_t') & =& \po b_2(X_t) - b_2(X_t')\pf \dd t +  2 \beta(X_t,X_t') e_t e_t^T \dd B_t  \,. 
\end{array}
\right.
\end{equation}

In the rest of this section, we consider some $\xi \in(0,1]$, and we denote by $C$ various constants which change from line to line, depending on the constants in Assumption~\ref{assum:decomposeYZ}  but not on $\xi$.

\medskip

\noindent\textbf{Step 2. Balance between the $Y$'s and the $Z$'s.} Set $y_t = |Y_t - Y_t'|$, $z_t = |Z_t -Z_t'|$. The goal of this step is to define a distance between the processes, equivalent to $y_t+z_t$, which is contracted by the drift $b$ when $y_t \gg z_t$.

 In order to apply It$\bar{\text{o}}$ formula, we consider a $\mathcal C^2$ approximation of the Euclidean norm. More specifically, consider a non-decreasing $h\in\mathcal C^2(\R_+,\R_+)$ with $h(x^2)=0$ for $x\in[0,\xi^2/4]$, $h(x^2) = x$ for $x\geqslant \xi/2$ and, for $x\in[\xi^2/4,\xi/2]$, $2xh'(x^2) \in[0, 1+4\xi ]$ and $h(x^2) \leqslant x$ (which can be achieved by taking as $h(x^2)$ for $x\in[\xi^2/4,\xi/2]$ a smooth approximation of the linear interpolation between $0$ at $x=\xi^2/4$ and $\xi/2$ at $x=\xi/2$, which on this interval  is strictly below $x\mapsto x$ and has derivative $(1-\xi/2)^{-1} < 1+4\xi$ since $\xi \in(0,1]$). In particular, for all $x\geqslant 0$, $2xh'(x^2) \leqslant 1+4\xi$ and $x \geqslant h(x^2) \geqslant x - \xi/2$.  Then, from~\eqref{eq:couplage2},
\begin{align*}
\dd h(y_t^2) &= 2 h'(y_t^2) (Y_t-Y_t') \cdot \po b_1(X_t)-b_1(X_t')\pf \dd t \\
& \leqslant  - \rho_1  y_t \1_{y_t\geqslant \xi/2} \dd t   + L_1 \po 1+ 4\xi\pf  z_t \dd t \\
& \leqslant - \rho_1  y_t   \dd t  + \rho_1 \xi /2 \dd t   + L_1 \po 1+ 4\xi\pf  z_t \dd t 
\end{align*}
and, using that $\beta_t:=\beta(X_t,X_t')$ is zero whenever $z_t \leqslant \xi/2$ and that $h(z_t^2)= z_t$ otherwise,
\begin{align*}
\dd h(z_t^2) &= 2 h'(z_t^2) (Z_t-Z_t') \cdot \po b_2(X_t)-b_2(X_t')\pf \dd t  + 4 \beta_t e_t^T \dd B_t \\
& \leqslant (1+4\xi) \po L_2 y_t + L_3 z_t\pf \dd t  + 4 \beta_t e_t^T \dd B_t\,.
\end{align*}
As a consequence, introducing the one-dimensional Brownian motion $W_t = \int_0^t e_s^T \dd B_s$ and
\[R_t = Mh(y_t^2) + h(z_t^2)\quad \text{with}\quad M = 2L_2/\rho_1\,,\]
which is such that
\begin{equation}
\label{loc:equivR}
My_t +z_t - (M+1)\xi/2 \leqslant  R_t\leqslant M y_t + z_t,,
\end{equation}
we have obtained that
\[\dd R_t = A_t \dd t + 4 \beta_t \dd W_t\]
with
\begin{align*}
A_t  & \leqslant  \po - \rho_1 M + L_2 + C\xi \pf   y_t    + \po L_3+ L_1 M  + C\xi\pf  z_t + C \xi  \\
& \leqslant -\frac{\rho_1}{2} M y_t + (L_3+L_1M) z_t +  C \xi (1+R_t)\,.
\end{align*}
As intended, we see that $A_t \leqslant - c R_t$ for some $c>0$ when $y_t\gg z_t$ (up to a vanishing term as $\xi \rightarrow 0$).

\medskip

\noindent\textbf{Step 3. Concave modification and exploitation of the noise.} Now, for $f \in \mathcal C^2(\R_+,\R_+)$ non-decreasing,  concave with $f'(0)=1$,
\[\partial_t \mathbb E \po f(R_t)\pf  \leqslant \mathbb E \po F(R_t) + C\xi (1+R_t)\pf\]
with
\[F(R_t) = f'(R_t) \co  -\frac{\rho_1}{2} M y_t + (L_3+L_1M) z_t \cf + 8 f''(R_t) \beta_t^2 \,.  \]
To go further, let us first partially specify $\beta$, that we take of the form
\[\beta(x,x') = \chi (|z-z'|^2) \alpha\po M h(|y-y'|^2) + h(|z-z'|^2)\pf \,,\]
for $\chi,\alpha\in\mathcal C^2(\R_+,\R_+)$, with $\chi(r^2)= 0$ for $r\leqslant \xi/2$, $\chi(r^2)\in[0,1]$ for $r\in[\xi/2,\xi]$ and $\chi(r^2)=1$ for $r\geqslant \xi$. The choice of $\alpha$ is left free for now. This form of $\beta$ ensures that
\[z_t\leqslant\frac{\xi}{2}\ \Rightarrow \ \beta_t=0\, \qquad \text{and} \qquad z_t\geqslant \xi \ \Rightarrow \ \beta_t = \alpha(R_t)\,.\] 
 Now we distinguish several cases.

\begin{itemize}
\item If 
\[ y_t \geqslant \frac{4(L_3+L_1M)}{\rho_1 M } z_t   \] 
then, using that $f$ is non-decreasing and concave, 
\begin{align*}
F(R_t) & = f'(R_t) \co  -\frac{\rho_1}{2} M y_t + (L_3+L_1M) z_t \cf + 8 f''(R_t) \beta_t^2   \\
& \leqslant -\frac{\rho_1}{4} M y_t  f'(R_t)   \\
& \leqslant - \rho_3 R_t f'(R_t)   + C \xi \,,\quad \text{with}\quad \rho_3= \frac{\rho_1}{4(1+\rho_1/[4(L_3+L_1M)])}\,.
\end{align*}
In this case, we have exploited the deterministic contraction given by the $Y$ variables when they are the main contribution in the distance between the two processes.
\item If 
\[ y_t < \frac{4(L_3+L_1M)}{\rho_1 M } z_t \quad \text{and}\quad z_t \leqslant \xi\,,   \] 
then $F(R_t)$, $R_t$ are all of order $\xi$, and then
\[ F(R_t) \leqslant - \rho_3 R_t f'(R_t)   + C \xi\,.\]
with the same $\rho_3$ as before. In this case, we have used that the two processes are very close and  thus we don't have anything sharp to say as the regularization error $\xi$ is dominant.
\item If 
\[ y_t < \frac{4(L_3+L_1M)}{\rho_1 M } z_t \quad \text{and}\quad z_t \geqslant \xi\,,   \] 
then
\begin{align*}
F(R_t) & \leqslant  f'(R_t)  (L_3+L_1M) R_t  + 8 \alpha^2(R_t) f''(R_t) \,.
\end{align*}
In this case, we have sufficient noise since $z_t$ is large, and thus we have to exploit it by taking $f''$ sufficiently negative.
\end{itemize}
In order to get in this third case the same bound as in the two others, we are led to chose $f$ as the solution of
\begin{equation}
\label{eq:eqdiff_f}
-\rho_3 r f'(r) =  f'(r)  (L_3+L_1M) r  + 8 \alpha^2(r) f''(r)\qquad \forall r\geqslant 0\,,
\end{equation} 
with $f(0)=0$ and $f'(0)=1$. In other words, writing $L_4=(L_3+L_1M +\rho_3)/8$,
\begin{equation}
\label{eq:eqdiff_f2}
f(r)= \int_0^r f'(s)\dd s\,,\qquad  f'(r) = \exp\po - L_4 \int_0^r \frac{s}{\alpha^2(s)} \dd s \pf\,.
\end{equation}
As will be clear in the next step, it will be useful to get $\alpha(r)=0$ for large values of $r$. As a consequence, given some $R_1^*>0$ to be specified later on, we chose $\alpha \in\mathcal C^2(\R_+,\R_+) $ so that 
\[\alpha(r)\left\{\begin{array}{cll}
= & \theta &  \text{for } r\in[0,R_1^*]\\
\in & (0,\theta] &  \text{for } r\in [R_1^*,R_1^*+\xi/2] \\
= & \theta (r - R_1^*-\xi)^2 &  \text{for } r\in [R_1^*+\xi/2,R_1^*+\xi]  \\
=& 0  & \text{for } r\geqslant R_1^*+\xi\,.
\end{array}
\right.\]
This choice ensures that $f'$ defined by \eqref{eq:eqdiff_f2} over $[0,R_1^*+\xi)$ and equal to zero over $[R_1^*+\xi,\infty)$ is $\mathcal C^1$ and satisfies~\eqref{eq:eqdiff_f}. From now on, $f$ is fixed to this definition. In particular, for $r\leqslant R_1^*$,
\begin{equation}
\label{locf'r}
f'(r) = \exp \po - \frac{L_4 r^2}{2 \theta^2}\pf \,.
\end{equation}
At this point, our choices have led to: 
\begin{equation}
\label{locf}
\partial_t \mathbb E \po f(R_t)\pf  \leqslant \mathbb E \po -\rho_3 R_t f'(R_t) + C\xi (1+R_t)\pf\,,
\end{equation}
and we have still one free parameter to chose, $R_1^*$  (which describes the value of $R_t$ above which the noise is switched off).

\medskip

\noindent\textbf{Step 4. Contraction at infinity.} Now, we use Assumption~\ref{assum:Contract_at_infty}. Setting $S_t = \|X_t-X_t'\|_Q$ and considering an increasing convex function $g\in\mathcal C^2(\R_+,\R_+)$ with $g(s) = 0$ for $s\leqslant S^*$, 
\begin{equation}
\label{locg}
\partial_t \mathbb E\po  g(S_t) \pf   \leqslant \mathbb E \po   - \rho_2 g'(S_t) S_t + 8 |Q_{22}|g''(S_t) \alpha^2(R_t)\pf   \,,
\end{equation}
with $Q_{22}$ the last $n\times n$ block of $Q$ (Notice that $|Q_{22}| \leqslant |Q|$ so in the following we may replace $|Q_{22}|$ by $|Q|$ everywhere if we want to have expressions involving only the parameters in $\mathfrak{P}$' without $|Q_{22}|$). For $\varepsilon>0$ and $S_2^* >  S^* $ to be chosen, we take $g$ so that
\[g''(s)\left\{\begin{array}{cll}
= & 0 & \text{for } s\leqslant S^*\\
\in & [0,\varepsilon] &  \text{for } s\in[S^*,S^*+\xi]\\
= &\varepsilon &  \text{for } s\in [S^*+\xi,S_2^*] \\
\geqslant  & 0  &  \text{for } r\in [S_2^*,S_2^*+\xi]  \\
=& p(p-1) \po r - S_2^*-\xi \pf ^{p-2} & \text{for } r\geqslant S_2^*+\xi\,,
\end{array}
\right.\]
with $g(0)=g'(0)=0$. Notice that, thanks to~\eqref{loc:equivR},
\begin{equation}
\label{loc12ER}
 \gamma_1 R_t  \ \leqslant \   S_t  \ \leqslant\  \gamma_2  R_t + C_0\xi \,,
\end{equation}
with $\gamma_1 = \min(M,1) |2Q^{-1}|^{-1/2}$, $\gamma_2 = \max(M^{-1},1)\sqrt{|Q|}$ and $C_0 = \gamma_2(M+1)/2$. We assume that $S_2^*$ and $R_1^*$ are such that $S_2^* \geqslant \gamma_2 R_1^* + (C_0+\gamma_2) \xi$, which ensures that 
\[\alpha(R_t) \neq 0 \quad \Rightarrow \quad R_t \leqslant R_1^* + \xi  \quad \Rightarrow \quad   S_t \leqslant S_2^* \quad \Rightarrow \quad g''(S_t) \leqslant \varepsilon\,.\]
Besides, since $g''(s) \geqslant 0$ for all $s\geqslant 0$ and $g''(s) = \varepsilon$ for all $s \in [S^*+\xi,S_2^*]$, $g'(s) \geqslant \varepsilon (\min(s,S_2^*)-S^*-\xi)$ for all $s\geqslant S^*+\xi$, and thus
\begin{equation}
\label{locS1}
 8|Q_{22}|\varepsilon\theta^2  \leqslant  \frac{\rho_2}2 g'(s) S^* \leqslant  \frac{\rho_2}2 g'(s) s\,, \qquad  \forall s \geqslant S^*+\xi  + \frac{16|Q_{22}|\theta^2}{S^*\rho_2}  =:S_1^*\,,
\end{equation}
assuming that $S_1^* \leqslant S_2^*$. As a consequence, enforcing the condition $S_2^* \geqslant \max(\gamma_2 R_1^* + (C_0+\gamma_2) \xi,S_1^*)$, we deduce from~\eqref{locg} that
\[\partial_t \mathbb E\po  g(S_t) \pf   \leqslant \mathbb E \po   - \frac{\rho_2}2 g'(S_t) S_t + 8 |Q_{22}|\varepsilon\theta^2 \1_{[S^*,S_1^*]}(S_t)\pf   \,. \]
In the case $S_t \leqslant S_1^*$, it remains to control the last term here thanks to the negative term in~\eqref{locf} which, using that
\begin{equation}
\label{locR0}
S_t \in [S^*, S_1^*] \qquad \Rightarrow \qquad R_t \in \co \frac{S^* -C_0\xi}{\gamma_2} ,\frac{S_1^*  }{\gamma_1} \cf=: [R^*,R_0^*]\,,
\end{equation}
 is done by assuming that $R_1^* \geqslant R_0^*$ (so that $f'(R_t)$ is given by the formula~\eqref{locf'r} for $R_t \leqslant R_0^*$, and in particular is larger than $f'(R_0^*)$)  and choosing
 \begin{equation}
 \label{locepsi}
\varepsilon = \frac{\rho_3 R^*f'(R_0^*)}{16|Q_{22}|\theta^2}\,,
 \end{equation}
 (assuming that $\xi < S_*/C_0$ so that $R^*>0$) to get that
\[  8 |Q_{22}|\varepsilon\theta^2 \1_{[S^*,S_1^*]}(S_t) \leqslant \frac{\rho_3}2 R_t f'(R_t)\,. \] 
To summarize, defining $S_1^*,R^*$ and $R_0^*$ by \eqref{locS1} and \eqref{locR0}, we choose $R_1^*= R_0^*$, $\varepsilon$ given by~\eqref{locepsi} and $S_2^* = \max(\gamma_2 R_1^* + (C_0+\gamma_2) \xi,S_1^*)$. With these choices, combining~\eqref{locf} and~\eqref{locg} gives
\begin{equation}
\label{locpER}
\partial_t \mathbb E\po f(R_t)+  g(S_t) \pf   \leqslant \mathbb E \co -\frac{\rho_3}2 R_t f'(R_t)  - \frac{\rho_2}2 g'(S_t) S_t + C\xi (1+R_t)   \cf  \,. 
\end{equation}
for $\xi<S^*/C_0$.

\medskip

\noindent\textbf{Step 5. Global contraction of a suitable cost.}  For $x=(y,z),x'=(y',z')\in\R^{n+m}$, set
\[s=\|x-x'\|_Q\,,\qquad r= M h \po |y-y'|^2\pf + h\po |z-z'|^2\pf\,.\]
  We distinguish three  cases.
\begin{itemize}
\item When $s \geqslant 2( S_2^*+\xi)$, 
\[g'(s) \geqslant p \po s - S_2^* - \xi\pf^{p-1} \geqslant p (s/2)^{p-1} \,,\]
so that 
\[-\frac{\rho_3}2 r f'(r)  - \frac{\rho_2}2 g'(s) s \leqslant - \frac{p \rho_2}{2^p} s^{p} \,. \]
On the other hand,
\begin{align*}
 f(r) + g(s) & \leqslant f(R_1^*+\xi) + g(S_2^* +\xi) + g'(S_2^*+\xi) (s-S_2^*-\xi) + (s -S_2^* - \xi)^p\\
 & \leqslant K_1  s^p
\end{align*}
with
\[K_1  = 1+ \frac{ f(R_1^*+\xi) + g(S_2^* +\xi) }{(S_2^*)^p} + \frac{g'(S_2^*+\xi)}{(S_2^*)^{p-1}}\,,  \]
and conversely
\begin{align*}
 f(r) + g(s) & \geqslant  g(S_2^* +\xi) + g'(S_2^*+\xi) (s-S_2^*-\xi) + (s -S_2^* - \xi)^p\\
 & \geqslant (s -S_2^* - \xi)^p\\
 & \geqslant (s/2)^p\,.
\end{align*}

\item  When $s \in[ S_1^*,2(S_2^*+\xi)]$, 
\[ \frac{\rho_2}2 g'(s)   \geqslant \frac{\rho_2}2 g'(S_*^1)   \geqslant  \frac{8|Q_{22}|\varepsilon\theta^2 }{S^*} =:c_0\,,  \]
so that 
\[-\frac{\rho_3}2 r f'(r)  - \frac{\rho_2}2 g'(s) s \leqslant - c_0 S_1^*   \,. \]
On the other hand,
\[ f(r) + g(s) \geqslant  g(S_1^*)  \geqslant  \frac{g(S_1^*)}{2(S_2^*+\xi) }s   \,, \] 
and
\[    f(r) + g(s) \leqslant f(R_1^*+\xi) + g\po 2(S_2^*+\xi+1)\pf \leqslant K_2 s  \,, \]
with $K_2 = [f(R_1^*+\xi) + g\po 2(S_2^*+\xi+1)\pf]/S_1^*$. 
\item When $s \leqslant S_1^*$, then $r \leqslant R_0^* \leqslant R_1^*$ and thus we can bound
\[-\frac{\rho_3}2 r f'(r)  - \frac{\rho_2}2 g'(s) s \leqslant   -\frac{\rho_3}2 r f'(R_1^*)\]
and, using that $f$ is concave with $f'(0)=1$ and that $g(s) = 0$ for  $s\leqslant S^* $,
\[f'(R_1^*) r \leqslant  f(r)+  g(s) \leqslant r +  \1_{[S^*,\infty)}(s) g(S_1^*) \leqslant r + \frac{g(S_1^*)}{S^*}s \leqslant K_3 r\,, \]
with $K_3 = 1+ \sqrt{|Q|}\max(M^{-1},1)g(S_1^*)/S^*$.
\end{itemize}
As a conclusion, we have obtained that in any of these three cases,
\[-\frac{\rho_3}2 r f'(r)  - \frac{\rho_2}2 g'(s) s \leqslant -c_* \po f(r)+  g(s)\pf \,,  \]
with
\begin{equation}\label{locc*}
c_* = \min \left\{  \frac{p \rho_2}{2^p K_1}\,,\,  \frac{c_0}{K_2}\,,\, \frac{\rho_3}{2K_3} f'(R_1^*)\right\} >0\,,
\end{equation}
and, writing $\omega(x,x')=\max(|x-x'|^p,|x-x|)$ and using the equivalence between $r$, $s$ and $|x-x'|$ (up to a term $C\xi$ when lower-bounding $r$ in terms of $|x-x'|$, as in~\eqref{loc:equivR}),
\[ k_* \omega(x,x') -C \xi  \leqslant f(r)+  g(s) \leqslant K_* \omega(x,x') \] 
with
\begin{align}
k_* &=  \min\Big\{		\po 2|Q^{-1/2}|\pf^p \min\po 1,(2S_2^*)^{p-1}\pf 	\,,\,	\frac{g(S_1^*)}{2(S_2^*+\xi) }\min \po |Q^{-1/2}| , \frac{|Q^{-1/2}|^p }{(2S_2^*+\xi)^{p-1}}\pf \,,\, \nonumber \\
& \hspace{9cm} 			f'(R_1^*)  \frac{\min(M^p,1)}{\max(1,(R_0^*)^{p-1})}  		\Big\} >0 \label{lock*1}\\
K_* &= \max\left\{ K_1|Q|^{p/2} \,,\,	K_2\sqrt{|Q|}	\,,\, \sqrt{2} K_3 \max(M,1)\right\}>0\,. \label{locK*2}
\end{align}
It is readily checked from their expressions that $c_*,k_*$ and $K_*$ have positive limits as $\xi\rightarrow 0$. In particular, $r\leqslant C \po f(r)+g(s)\pf$ for some constant $C>0$ independent from $\xi\in(0,1]$.

Using these bounds in~\eqref{locpER} gives
\[\partial_t \mathbb E\po f(R_t)+  g(S_t) \pf  \leqslant -  \po c_* - C\xi\pf \mathbb E\po f(R_t)+  g(S_t) \pf  + C\xi   \,.   \] 

\medskip

\noindent\textbf{Step 6. Conclusion.} Given two initial distributions $\nu,\mu \in \mathcal P_p(\R^{d})$, consider the previous Markovian coupling with initial condition $(X_0,X_0')  \sim \pi \in \Pi(\nu,\mu)$. Then
\begin{align*}
\mathcal W_{\omega}(\nu P_t,\mu P_t) &\leqslant \mathbb E \po \omega(X_t,X_t')\pf \\
&\leqslant \frac{1}{k_*}  \mathbb E\po f(R_t)+  g(S_t) \pf  + C\xi \\
&\leqslant \frac{1}{k_*}  e^{-t(c_*-C\xi)} \mathbb E\po f(R_0)+  g(S_0) \pf  + e^{tC\xi} C\xi \\
&\leqslant \frac{K_*}{k_*}  e^{-t(c_*-C\xi)} \mathbb E\po \omega(X_0,X_0') \pf  + e^{tC\xi} C\xi \,.
\end{align*}
Taking the infimum over $\pi \in \Pi(\nu,\mu)$ and then letting $\xi \rightarrow 0$ concludes the proof of Theorem~\ref{thm:contractionMarkov}, since $\mathcal W_{\omega}(\nu P_u,\mu P_u)$ for $u\in\{0,t\}$ does not depend on $\xi$.

\medskip

\noindent\textbf{Explicit constants.} For the convenience of interested readers, we summarize here the explicit expression of $C_p$ and $\lambda_p$ established in the proof of Theorem~\ref{thm:contractionMarkov} (which are obtained by letting $\xi\rightarrow 0$ in the expressions given in the computations above). We have
\[C_p = \frac{K_*}{k_*}\qquad \lambda_p = c_* \,,\]
with $k_*,K_*$ and $c_*$ respectively given in~\eqref{lock*1}, \eqref{locK*2} and~\eqref{locc*} (with $\xi=0$), involving the following constants: 
\begin{align*}
M & = \frac{2L_2}{\rho_1}  & \rho_3 &= \frac{\rho_1}{4(1+\rho_1/[4(L_3+L_1M)])}   & S_1^* &= S^* + \frac{16|Q_{22}|\theta^2}{S^*\rho_2} \\
\gamma_1 &= \min(M,1) |2Q^{-1}|^{-1/2} & \gamma_2 &= \max(M^{-1},1)\sqrt{|Q|} & C_0 &= \gamma_2(M+1)/2\\
 R^* &= \frac{S^* }{\gamma_2} & R_1^*&  = \frac{S_1^*  }{\gamma_1} & S_2^* &= \max(\gamma_2 R_1^* ,S_1^*) \\
\varepsilon &= \frac{\rho_3 R^*}{16|Q_{22}|\theta^2} f'(R_1^*) & f'(R_1^*) &= \exp \po - \frac{L_4 (R_1^*)^2}{2 \theta^2}\pf  &  c_0 &= \frac{8|Q_{22}|\varepsilon\theta^2 }{S^*}  
\end{align*}
and
\begin{align*}
K_1  &= 1+ \frac{ f(R_1^*) + g(S_2^*) }{(S_2^*)^p} + \frac{\varepsilon(S_2^*-S^*)}{(S_2^*)^{p-1}} 
\qquad
 K_2 = [f(R_1^*) + g\po 2(S_2^*+1)\pf]/S_1^*& & \\
K_3 &= 1+ \sqrt{|Q|}\max(M^{-1},1)g(S_1^*)/S^*
\end{align*}
and, finally, the functions
\[f(r)= \int_0^{\min(r,R_1^*)} \exp\po - \frac{L_4  s}{2\theta^2}  \pf \dd s \,,\qquad
g(s) = \left\{\begin{array}{ll}
  0 & \text{for } s\leqslant S^*\\
 \frac{\varepsilon}2 (s-S^*)^2 &  \text{for } s\in [S^*,S_2^*] \\
  \po s - S_2^* \pf ^{p} +\frac{\varepsilon}2 (S_2^*-S^*)^2 & \text{for } r\geqslant S_2^*\,.
\end{array}
\right.\]

\subsection{Concentration inequalities  }\label{subsec:concentration}

Since Theorem~\ref{thm:contractionMarkov} applied with $p=1$ gives the contraction~\eqref{eq:contraction_def}, let us state a general consequence of such a contraction, following similar results in \cite{Djellout,LMM} and references within.

A probability measure $\mu \in \mathcal P(\R^d)$ is said to satisfy a $L^1$-transportation cost-information inequality with constant $C > 0$ (which we write: $\mu$ satisfies $T_1(C)$) if for all $\nu \ll \mu$,
\[\mathcal W_1(\nu,\mu) \leqslant \sqrt{2C \mathcal H(\nu|\mu)}\,,\]
with the relative entropy $\mathcal H(\nu|\mu) = \int_{\R^d} \ln\frac{\dd \nu}{\dd \mu} \dd \mu$. According to \cite[Theorem 2.3]{Djellout}, this is equivalent to a Gaussian tail condition, i.e. the fact that $\int_{\R^d} e^{\delta|x|^2}\mu(\dd x) <\infty$ for some $\delta>0$.  By~\cite{otto2000generalization} (and since $\mathcal W_1\leqslant \mathcal W_2$), it is also implied by a log-Sobolev inequality~\eqref{eq:LSI}. 

\begin{proposition}\label{prop:concentration}
Under Assumptions~\ref{Assum:basic}, assume moreover that the contraction~\eqref{eq:contraction_def} holds with $p=1$ for some $C,\lambda>0$. Denote by $\mu_\infty$ the invariant measure of the diffusion process~\eqref{eq:EDS_X} (whose existence and uniqueness follow from~\eqref{eq:contraction_def}). Let $X$ solve~\eqref{eq:EDS_X} with $X_0\sim \nu_0$. Assume that $\nu_0$ satisfies $T_1(C_0)$ for some $C_0>0$. Then, for all Lipschitz $\varphi:\R^d\rightarrow \R$ with Lipschitz constant $L_\varphi>0$, and all $t>0$,
 \begin{equation}
 \label{loc:concentreExpo}
\mathbb E_{\nu_0} \po \exp \po  \int_0^t \co \varphi(X_s) - \mathbb E\po \varphi(X_s)\pf \cf  \dd s  \pf\pf \leqslant   \exp\left(  \po C_0+  2 t |\sigma|^2 \pf   \frac{ C^2 L_\varphi^2 }{2 \lambda^2}\right)\,,
 \end{equation}
 and
  \begin{equation}
 \label{loc:concentreExpo2}
\mathbb E_{\nu_0} \po \exp \po  \int_0^t   \varphi(X_s)   \dd s  \pf\pf \leqslant  A e^{r t}  \,,
 \end{equation}
 with
 \[A= \exp\left(   \frac{C_0 C^2  L_\varphi^2 }{2 \lambda^2}  +   \frac{ L_\varphi C}{\lambda} \mathcal W_1(\nu,\mu_\infty) \right)\,, \qquad		r =   \mu_\infty(\varphi) +   \frac{  |\sigma|^2 C^2  L_\varphi^2 }{\lambda^2} \,.\]
\end{proposition}
By the Markov inequality, this gives concentration bounds for empirical averages $\frac1t\int_0^t \varphi(X_s)\dd s$. However, in the present study, in order to bound $G_p(x,t)$, what we need is directly~\eqref{loc:concentreExpo2}.

\begin{proof}
 Similar results have been established in \cite{Djellout}. A simple way to get the proposition exactly as stated in to follow  \cite[Proposition 7.1]{LMM} (which is based on \cite{Djellout}) applied to the discrete time chain $(Y_n)_{n\in\N}=(X_{n\delta})_{n\in\N}$ for some step-size $\delta>0$ and then to send $\delta$ to zero. Indeed, denoting by $Q=P_\delta$ the Markov transition operator associated to $Y$, we get that
\[ \forall k\in\N,\ \nu,\mu \in\mathcal P_1(\R^d)\,,\qquad \mathcal W_1(\nu Q^k,\mu Q^k) \leqslant C e^{- \lambda k \delta} \mathcal W_1(\nu,\mu)\,.\]
Besides, the one-sided condition on $b$ implies by \cite[Proposition 2]{MonmarcheContraction} that, for any $x\in\R^d$, $\delta_x Q$ satisfies a log-Sobolev inequality with constant $|\sigma|^2 ( e^{2L\delta}-1)/L$, which implies a $T_2$ and hence a $T_1$ inequality with the same constant~\cite{otto2000generalization}. As a consequence, as established in the proof of \cite[Proposition 7.1]{LMM}, for any $n\in\N$ and $a \geqslant 0$,
\[\mathbb E \po \exp \po a \sum_{k=0}^{n-1} \co \varphi(Y_k)-\mathbb{E}\po  \varphi(Y_k) \pf  \cf   \pf\pf \leqslant   \exp\left( \po \frac{(n-1) |\sigma|^2 ( e^{2L\delta}-1)}L + C_0\pf   \frac{ C^2 a^2L_\varphi^2}{2(1-e^{-\lambda\delta})^2}\right). \] 
Applying this with $a =\delta$ and $n = \lceil t/\delta\rceil$ for a fixed $t>0$ and letting $\delta\rightarrow 0$ gives~\eqref{loc:concentreExpo}.

Moreover, for all $t\geqslant 0$,
\[|\mathbb E \po \varphi(X_t)  \pf - \mu_\infty (\varphi)| \leqslant L_\varphi \mathcal W_1(\nu P_t,\mu_\infty) \leqslant L_\varphi C e^{-\lambda t} \mathcal W_1(\nu,\mu_\infty)\,. \]
Thus, denoting by $C_t$ the right hand side of~\eqref{loc:concentreExpo}, we can bound
\begin{align*}
\mathbb E \po \exp \po  \int_0^t  \varphi(X_s)    \dd s  \pf\pf & \leqslant C_t  \exp \po  \int_0^t  \mathbb E\po \varphi(X_s)\pf  \dd s  \pf \\
& \leqslant 
  C_t \exp \po   t \mu_\infty (\varphi) +  \int_0^t  |\mathbb E \po \varphi(X_s)  \pf - \mu_\infty (\varphi)| \dd s\pf \\ 
 & \leqslant  C_t \exp \po   t \mu_\infty (\varphi) +    \frac{L_\varphi C}{\lambda} \mathcal W_1(\nu,\mu_\infty) \pf \,,
\end{align*}
which is~\eqref{loc:concentreExpo2}.
%
\end{proof}

\begin{corollary}
Under the settings of Proposition~\ref{prop:concentration}, for any $R,t>0$ and Lipschitz function $\varphi$ with Lipschit constant $L_\varphi$,
  \begin{equation}
 \label{loc:concentreExpo3}
\sup_{x\in \mathcal B(0,R)} \mathbb E_{x} \po \exp \po  \int_0^t   \varphi(X_s)   \dd s  \pf\pf \leqslant  A e^{r t}  \,,
 \end{equation}
 with
 \[A= \exp\left(      \frac{ L_\varphi C}{\lambda} \po R + \int_{\R^d} |y| \mu_\infty(\dd y)\pf  \right)\,, \qquad		r =  \mu_\infty(\varphi) +   \frac{  |\sigma|^2 C^2  L_\varphi^2 }{\lambda^2} \,.\]
\end{corollary}

\begin{proof}
Apply Proposition~\ref{prop:concentration} with $\nu = \delta_x$ (which satisfies $T_1(C_0)$ with $C_0=0$) when $|x|\leqslant R$, bounding 
\[\mathcal W_1(\delta_x,\mu_\infty) \leqslant R+ \int_{\R^d} |y| \mu_\infty(\dd y)\,.\]
\end{proof}

We can now conclude as follows:

\begin{proof}[Proof of Proposition~\ref{prop:concentreGp*}]
This is the same arguments as in the proof of~\eqref{eq:limsupGJpeta}, except that we use~\eqref{loc:concentreExpo3} (with $\varphi=p \eta$) instead of~\eqref{locGCR} in~\eqref{locFGH}.
\end{proof}

\subsection{Proof of Theorem~\ref{thm:final}}\label{subsec:demofinal}

%
%

\begin{proof}[Proof of Theorem~\ref{thm:final}]
First, let us notice that, under the settings of Theorem~\ref{thm:final}, by applying Theorem~\ref{thm:contractionMarkov} with $p=1$ and using the bounds~\eqref{locboundparam} we get that the contraction~\eqref{eq:contraction_def} holds with $p=1$ for some constants $C=C_1$ and $\lambda=\lambda_1$ depending only on $\unif{L},\unif{\rho},\unif{S},\unif{M}$ and $\unif{\theta}$. In particular, as discussed before Proposition~\ref{prop:concentreGp*}, the process~\eqref{eq:EDS_X} admits a unique invariant measure $\mu_\infty$.

 We do a linear transform $Y_t = Q^{1/2} X_t$ and write $(P_t^Q)_{t\geqslant 0}$ the associated Markov semi-group. Using the equivalence between distances
\[|Q^{1/2}|^{-1} |Q^{1/2}(x-x')| \leqslant |x-x'| \leqslant |Q^{-1/2}||Q^{1/2}(x-x')|\qquad \forall x,x'\in\R^d\,,\]
we deduce that $(P_t^Q)_{t\geqslant 0}$ also satisfies the contraction~\eqref{eq:contraction_def} for $p=1$ with $C = \unif{M}C_1$ and $\lambda=\lambda_1$. Similarly, it is sufficient to conclude the proof of Theorem~\ref{thm:final} for $P_t^Q$ to deduce the same conclusion for $P_t$. 

The process $Y_t$ solves
\[\dd Y_t = b_Q(Y_t) \dd t + \sigma_Q \dd B_t\]
with $b_Q(y) = Q^{1/2} b(Q^{-1/2}y)$ and $\sigma_Q = Q^{1/2} \sigma$. Setting $y=Q^{1/2}x$ and $w=Q^{1/2}v$, the condition~\eqref{locvQv}  implies that, 
 for all $y \in \R^d$ and $w\in\Sd$,
\begin{equation*}
w\cdot  \na b_Q(y) w \leqslant \left\{\begin{array}{ll}
\unif{\delta}&  \text{if }  y \in  Q^{1/2}  \mathcal B(0, \unif{S}) \\
- \unif{\rho} & \text{if }y \notin Q^{1/2}  \mathcal B(0, \unif{S})\,.
\end{array}\right.
\end{equation*}
Taking $\unif{\delta}$ small enough so that $\unif{\delta} \leqslant \unif{\rho} \unif{q}$, Assumption~\ref{Assum:PointwiseContractInfty} is thus satisfied (for the process $Y_t$) with the function $\eta$ given by
\begin{equation}
\label{locvQv2}
\eta(y) =\left\{\begin{array}{ll}
\unif{\delta}&  \text{for } y \in  Q^{1/2}  \mathcal B(0, \unif{S})  \\
\unif{\delta} \co 1+   \po 1  - \frac{1-1/\unif{q}}{\unif{S_2}-\unif{S}}\pf \po |Q^{-1/2}y| - \unif{S} \pf     \cf   & \text{for } y \in Q^{1/2}  \mathcal B(0, \unif{S_2})\setminus Q^{1/2}  \mathcal B(0, \unif{S})\\ 
- \unif{\delta}/\unif{q} & \text{for }  y \notin Q^{1/2}  \mathcal B(0, \unif{S_2})\,.
\end{array}\right.
\end{equation}
 We can thus apply Proposition~\ref{prop:concentreGp*} which, thanks to~\eqref{loc:betakappaG} concludes the proof of Theorem~\ref{thm:final} \emph{provided $\rho'>0$ in~\eqref{locArho'}}. However, thanks to~\eqref{eq:muinftyq}
 \[\mu_\infty(\eta) \leqslant (1-\unif{q}) \unif{\delta} -  \unif{q} \frac{\unif{\delta}}{\unif{q}}  = - \unif{q} \unif{\delta}\,. \]
 As a consequence, 
 \[\rho' \geqslant \unif{a} \unif{\delta} - \unif{b} (\unif{\delta})^2\]
 for some $\unif{a},\unif{b}>0$ which depend only on $\unif{\rho},\unif{\rho}',\unif{S},\unif{M},\unif{\theta}$ and $p$. Taking $\unif{\delta} $ small enough concludes the proof.

\end{proof}

\section{Conclusion}\label{sec:conclusion}

Along this work,  several questions have been raised and left open. Let us  summarize them:
\begin{enumerate}
\item Does the Brownian motion on the periodic torus induces an $L^p$ Wasserstein contraction for $p>1$ (see Remark~\ref{rem:torus})  ? 
\item For a given $p>1$ and a non-convex potential $U\in\mathcal C^2(\R^d,\R)$, is there a temperature $\theta_*^2>0$ such that for $\theta \in(0,\theta_*]$, the overdamped Langevin process~\eqref{eq:overdampedSDE} does \emph{not} induce an $L^p$ Wasserstein contraction (i.e. $\beta_p(t) \geqslant 1$ for all $t\geqslant 0$) ? Same question but only along synchronous couplings (i.e. $\kappa_p(t) \geqslant 1$ for all $t\geqslant 0$) ? (see Figure~\ref{fig:Lyap})
\item Is the result of  \cite{MonmarcheBruit} still true in the kinetic  case ? Namely, for the kinetic Langevin process~\eqref{eq:kinLangevin} with given friction $\gamma>0$ and potential $U$, strongly convex outside a ball and such that $\|\na^2 U\|_\infty \leqslant \gamma^2$ (from which Assumption~\ref{Assum:PointwiseContractInfty} holds up to a suitable linear change of variable according to \cite[Proposition 4]{MonmarcheContraction}), and for $p>1$, does it exist $\theta_*>0$ such that the process induces an $L^p$ Wasserstein contraction whenever $\theta \geqslant \theta_*$ ?
\end{enumerate}
Apart from these specific questions, there are several general directions in which the present analysis could be further developed.

 A first one would be to follow the discussion in Section~\ref{subsec:sharperDecompose} and in particular to use the representation~\eqref{eq:locaLyap} with a large deviation result for the process $(X_t,\theta_t)_{t\geqslant 0}$ to end up with a result similar to Theorem~\ref{thm:FerreStoltz} but in terms of $\kappa_p(t)$ instead of the bound $G_p^*(t)$ (leading to a variational formulation of the Lyapunov exponent $\Lambda_p^*$ in~\eqref{eq:Lambdap*}).

As a second direction of interest, a usual extension of coupling  arguments is to consider systems of interacting particles (for instance in the mean-field scaling), in the perturbative regime where the interaction is small enough with respect to the contraction rate, as for instance in \cite{HerdaPerthameMonmarche,LMM,schuh2024global,MonmarcheBruit}. As mentioned in the introduction, Wasserstein contraction in this situation is used to get uniform-in-time propagation of chaos, and uniform-in-$N$ functional inequalities~\cite{10.1214/24-EJP1217}. This would be a natural situation for applying our analysis, ideally to get general results in the high-temperature regime (long-time convergence to equilibrium, and thus a fortiori Wasserstein contraction, being not expected to be uniform in $N$ in the low-temperature regime in general situations, see \cite{2025arXiv250300157M} and references within).

Finally, the characterization of the contraction in terms of the Lyapunov exponent of a linearized problem can be generalized to settings beyond diffusion processes, including discrete space Markov chains.

\bibliographystyle{plain}
\bibliography{biblio} 

\end{document}